\pdfoutput=1
\documentclass{scrartcl}
\usepackage[intlimits]{amsmath}
\usepackage{amsthm,amssymb,amscd,bm,ifpdf,authblk,bbm,extpfeil,mathdots}
\usepackage{url}
\usepackage[abbrev]{amsrefs}
\usepackage[T1]{fontenc}
\usepackage{dutchcal}
\usepackage{xfrac}
\usepackage{lmodern}
\usepackage[inline]{enumitem}
\usepackage{microtype}
\usepackage[usenames,dvipsnames]{color}
\usepackage{tocstyle}
\usepackage{hyperref}

\allowdisplaybreaks[1]

\makeatletter
\DeclareOldFontCommand{\rm}{\normalfont\rmfamily}{\mathrm}
\DeclareOldFontCommand{\sf}{\normalfont\sffamily}{\mathsf}
\DeclareOldFontCommand{\bf}{\normalfont\bfseries}{\mathbf}
\DeclareOldFontCommand{\it}{\normalfont\itshape}{\mathit}
\makeatother

\newtheorem{thm}{Theorem}[section]
\newtheorem*{thm*}{Theorem}
\newtheorem{lemma}[thm]{Lemma}
\newtheorem*{lemma*}{Lemma}
\newtheorem{prop}[thm]{Proposition}
\newtheorem*{prop*}{Proposition}

\newtheorem{qu}[thm]{Question}
\newtheorem*{conj*}{Conjecture}
\newtheorem{cor}[thm]{Corollary}

\theoremstyle{definition}
\newtheorem{ex}[thm]{Example}
\newtheorem*{ex*}{Example}

\newtheorem{defn}[thm]{Definition}
\newtheorem*{defn*}{Definition}

\newtheorem{rem}[thm]{Remark}
\newtheorem*{rem*}{Remark}

\numberwithin{equation}{section}
\allowdisplaybreaks[1]

\renewcommand{\le}{\leqslant}
\renewcommand{\ge}{\geqslant}

\def\emph{}
\DeclareTextFontCommand{\bfemph}{\bf}
\DeclareTextFontCommand{\itemph}{\it}
\def\emph{\bfemph}

\newcommand{\stacks}[1]{{\cite[\href{http://stacks.math.columbia.edu/tag/#1}{Tag #1}]{stacks-project}}}

\makeatletter
\def\blankfootnote{\xdef\@thefnmark{}\@footnotetext}

\newcommand*{\textlabel}[2]{%
  \edef\@currentlabel{#1}
  \phantomsection
  #1\label{#2}
}
\makeatother

\newcommand{\RFL}{\textsf{\textup{(RFL)}}}
\newcommand{\STB}{\textsf{\textup{(STB)}}}

\newcommand{\onto}{\twoheadrightarrow}

\newcommand{\incl}{\hookrightarrow}
\newcommand{\xonto}{\xtwoheadrightarrow}

\newcommand{\xto}{\xrightarrow}

\newcommand{\cc}{\ensuremath{\mathsf{cc}}}

\newcommand{\mul}[1]{*_{{#1}}}
\newcommand{\opp}{{\operatorname{op}}}

\newcommand{\MV}{\ensuremath{\circ}}
\newcommand{\MW}{\ensuremath{\bullet}}
\newcommand{\VW}{\ensuremath{\vee}}

\newcommand{\up}{\ensuremath{\uparrow}}
\newcommand{\dn}{\ensuremath{\downarrow}}
\newcommand{\uuu}{{\ensuremath{\up\up\up}}}
\newcommand{\ddd}{{\ensuremath{\dn\dn\dn}}}
\newcommand{\ddu}{{\ensuremath{\dn\dn\up}}}
\newcommand{\udd}{{\ensuremath{\up\dn\dn}}}

\newcommand{\dud}{{\ensuremath{\dn\up\dn}}}

\newcommand{\sk}[1]{\Lambda({{#1}})}

\newcommand{\RF}{\mathfrak{K}}

\newcommand{\Mod}[1]{\mathrm{M}_{{#1}}}
\newcommand{\Dom}[1]{\mathrm{V}_{{#1}}}
\newcommand{\Cod}[1]{\mathrm{W}_{{#1}}}

\newcommand{\GRP}{\mathsf{Grp}}
\newcommand{\CALG}{\mathsf{CAlg}}
\newcommand{\MREP}{\mathsf{mrep}}
\newcommand{\MOD}{\mathsf{Mod}}
\newcommand{\MODf}{\mathsf{mod}}
\newcommand{\OBJ}{\mathsf{obj}}

\newcommand{\QQ}{\mathbf{Q}}

\newcommand{\GG}{\mathbf{G}}

\newcommand{\ZZ}{\mathbf{Z}}
\newcommand{\CC}{\mathbf{C}}
\newcommand{\RR}{\mathbf{R}}
\newcommand{\fa}{\ensuremath{\mathfrak a}}

\newcommand{\fg}{\ensuremath{\mathfrak g}}

\newcommand{\fz}{\ensuremath{\mathfrak z}}
\newcommand{\Places}{\ensuremath{\mathcal V}}

\newcommand{\Zeta}{\ensuremath{\mathsf{Z}}}

\newcommand{\PID}{\textsf{\textup{PID}}}

\newcommand{\DVR}{\textsf{\textup{DVR}}}
\newcommand{\QF}{\textsf{\textup{QF}}}

\newcommand{\ww}{\ensuremath{\bm w}}
\newcommand{\xx}{\ensuremath{\bm x}}

\newcommand{\zz}{\ensuremath{\bm z}}
\newcommand{\fp}{\mathfrak{p}}
\newcommand{\fo}{\mathfrak{o}}
\newcommand{\fO}{\mathfrak{O}}
\newcommand{\fP}{\mathfrak{P}}
\newcommand{\cA}{\mathcal{A}}

\newcommand{\sG}{\mathsf{G}}
\newcommand{\sH}{\mathsf{H}}

\newcommand{\cC}{\mathcal{C}}
\newcommand{\cX}{\mathcal{X}}
\newcommand{\cY}{\mathcal{Y}}

\DeclareMathOperator{\Irr}{Irr}

\DeclareMathOperator{\concnt}{k}

\DeclareMathOperator{\Sym}{Sym}
\DeclareMathOperator{\sym}{S}

\DeclareMathOperator{\comm}{D}

\newcommand{\Types}{\ensuremath{\mathsf{Tps}}}

\newcommand{\XX}{\bm{X}}

\DeclareMathOperator{\GL}{GL}

\DeclareMathOperator{\Ann}{Ann}

\newcommand{\Gl}{\ensuremath{\mathfrak{gl}}}

\newcommand{\So}{\ensuremath{\mathfrak{so}}}

\newcommand{\eval}[2]{\ensuremath{\varepsilon^{{#1}}_{{#2}}}}

\DeclareMathOperator{\KER}{\mathbcal{Ker}}
\DeclareMathOperator{\COKER}{\mathbcal{Coker}}

\DeclareMathOperator{\Ker}{Ker}
\DeclareMathOperator{\Coker}{Coker}

\DeclareMathOperator{\Cent}{C}

\DeclareMathOperator{\End}{End}
\DeclareMathOperator{\Hom}{Hom}

\DeclareMathOperator{\Mat}{M}

\DeclareMathOperator{\ad}{ad}

\DeclareMathOperator{\dd}{d\!}

\DeclareMathOperator{\kersize}{K}
\DeclareMathOperator{\imgsize}{I}
\DeclareMathOperator{\orbsize}{O}

\newcommand{\normal}{\triangleleft}
\newcommand{\dtimes}{\ensuremath{\,\cdotp}}
\newcommand{\card}[1]{\lvert#1\rvert}

\DeclarePairedDelimiter{\abs}{\lvert}{\rvert}
\DeclarePairedDelimiter{\norm}{\lVert}{\rVert}

\DeclareMathOperator{\rank}{rk}

\DeclareMathOperator{\genrank}{grk}

\DeclareMathOperator{\Real}{Re}

\DeclareMathOperator{\Img}{Im}

\newcommand{\llb}{\ensuremath{[\![ }}
\newcommand{\rrb}{\ensuremath{]\!] }}

\newcommand{\ask}[1]{\operatorname{ask}({#1})}
\newcommand{\askm}[2]{\operatorname{ask}^{{#2}}({#1})}
\newcommand{\Ask}{\text{\textup{ask}}}
\newcommand{\Askm}[1]{\text{\textup{ask}$^{{#1}}$}}

\newcommand{\coll}[2]{{}^{{#2}}\hspace*{-.1em}{#1}}

\ifpdf
  \hypersetup{pdftitle={The average size of the kernel of a matrix and orbits
      of linear groups, II: duality},
    pdfauthor={Tobias Rossmann}
  }
\fi
\title{The average size of the kernel of a matrix and orbits of linear groups,
  II: duality}
\author{Tobias Rossmann}
\affil{\small School of Mathematics, Statistics and Applied Mathematics\\ National
  University of Ireland, Galway}
\date{}

\begin{document}

\maketitle
\thispagestyle{empty}

\vspace*{-4em}
\begin{abstract}
   \small
   Define a module representation to be a linear parameterisation of
   a collection of module homomorphisms over a ring.
   Generalising work of Knuth, we define duality functors indexed by the
   elements of the symmetric group of degree three between categories of
   module representations. 
   We show that these functors have tame effects on average sizes of kernels.
   This provides a general framework for and a generalisation of
   duality phenomena previously observed in work of O'Brien and Voll
   and in the predecessor of the present article.
   We discuss applications to class numbers and conjugacy class zeta functions
   of $p$-groups and unipotent group schemes, respectively.
\end{abstract}

 \blankfootnote{%
   \noindent{\itshape 2010 Mathematics Subject Classification.}
   20D15, 20E45, 05A15, 15B33, 13E05, 11M41

  \noindent {\itshape Keywords.} 
  Average size of a kernel, duality, conjugacy classes, $p$-groups, Lie
  rings, unipotent groups

  \medskip
   {\noindent
     The author gratefully acknowledges the support of the
     \href{https://www.humboldt-foundation.de}{Alexander von Humboldt Foundation}.
   Substantial parts of the research described here were carried out while the
   author was a Feodor Lynen Research Fellow at the University of Auckland.
 }}

\tableofcontents

\section{Introduction}
\label{s:intro}

\paragraph{Class numbers.}
Let $\concnt(G)$ denote the class number (= number of conjugacy classes) of
a finite group $G$.
As is well-known,
\begin{equation}
  \label{eq:kIrr}
  \concnt(G) = \#\Irr(G)
  \tag{$\star$}
\end{equation}
is the number of ordinary irreducible characters of $G$.
No natural bijection between conjugacy classes and irreducible characters of
general finite groups has been found. 
Instead, \eqref{eq:kIrr} seems to be best understood in
terms of dualities; see e.g.\ the discussion~\cite{MO102879}.

In a recent paper~\cite{O'BV15}, O'Brien and Voll applied techniques
previously employed in the study of representation growth of infinite groups
to the study of class numbers of finite $p$-groups;
here, and throughout this article, $p$ denotes a prime.
The main techniques they used were the Lazard correspondence and the Kirillov orbit method.
The former provides an equivalence of categories between $p$-groups and 
Lie rings subject to suitable assumptions.
The latter provides a Lie-theoretic characterisation of irreducible
representations of suitable groups in terms of co-adjoint orbits.
Suppose that $G = \exp(\fg)$ is a finite $p$-group corresponding to a Lie ring
$\fg$ via the Lazard correspondence.
As a refinement of \eqref{eq:kIrr}, O'Brien and Voll obtained two different
expressions for $\concnt(G)$, corresponding to both sides of \eqref{eq:kIrr}; for
groups of exponent $p$, their description involves the rank loci of two types
of ``commutator matrices'' attached to $\fg$.

\paragraph{Average size of kernels.}
Given finite abelian groups $V$ and $W$ and a subgroup $M \subset
\Hom(V,W)$, the \underline average \underline size of the \underline kernel of
the elements of $M$ is the rational number
\begin{equation}
  \label{eq:askM}
  \ask M := \frac 1 {\card M} \sum_{a \in M}\card{\Ker(a)}.
  \tag{$\dagger$}
\end{equation}
These numbers appeared in work of Linial and Weitz~\cite{LW00} and Fulman
and Goldstein~\cite{FG15}.
They were further investigated by the author~\cite{ask}.
Using $p$-adic Lie theory, it turns out that numbers of the form $\ask M$
include numbers of orbits and conjugacy classes of suitable linear groups.
In \cite[\S 8]{ask}, this was applied to the study of orbit-counting and
conjugacy class zeta functions associated with unipotent pro-$p$ groups.
An elementary yet crucial ingredient of \cite{ask} was the following
observation (see \cite[Lem.\ 2.1]{ask}) going back to Linial and Weitz:
\begin{equation}
  \label{eq:askMdual}
  \ask M = \sum_{x\in V} \card{xM}^{-1}.
  \tag{$\ddagger$}
\end{equation}

\paragraph{Knuth duality.}
The two expressions~\eqref{eq:askM}--\eqref{eq:askMdual} for $\ask M$
and those for $\concnt(G)$ in \cite[Thm~A]{O'BV15}
are formally similar and seem to suggest a common underlying notion of duality.
Indeed, we will see in \S\S\ref{s:duality}--\ref{s:S3_effects} that a
suitable generalisation (see 
\S\ref{s:ask}) of the above setting for average 
sizes of kernels explains the dualities of both \cite{O'BV15} and \eqref{eq:askMdual}.
Namely, these dualities reflect the effects of two different
involutions in the symmetric group $\sym_3$ of degree three acting by functors
(which we call ``Knuth  dualities'') on the isomorphism
classes of objects of suitable categories of ``module representations''.
Special cases of this action are ubiquitous in the study of semifields where
they are known under the heading of ``Knuth orbits''; see \cite{Knu65b} and e.g.\ \cite{LP14}. 

While somewhat technical to set up rigorously in the present setting,
Knuth duality as studied here only requires elementary linear
algebra. The innocuous yet crucial abstraction step is the following: rather
than a subgroup $M$ of $\Hom(V,W)$ as above, we consider a possibly
non-injective homomorphism $M \xto{\theta} \Hom(V,W)$;
we will dub such maps ``module representations'' in \S\ref{s:mrep}.
Generalising \eqref{eq:askM},
we let 
\[
\ask\theta := \frac 1{\card M}\sum_{a\in M}\card{\Ker(a\theta)}.
\]

The aforementioned $\sym_3$-action now permutes the three groups $M$, $V$,
and $W$, combined with taking suitable duals.
While essentially equivalent $\sym_3$-actions have been studied in
related settings in the literature, the present article seems to be the first
to systematically explore the effects of this action on average sizes of
kernels.
In particular, our main result here, Theorem~\ref{thm:ask_duality},
constitutes a simultaneous and unified generalisation of both
\eqref{eq:askMdual} and the final assertion of \cite[Thm~A]{O'BV15}
(the aforementioned two expressions for $\concnt(G)$).

\paragraph{Lie algebras as module representations.}
Each Lie algebra $\fg$ over $\ZZ$ (``Lie ring'') gives rise to
a canonical associated module representation, namely its adjoint
representation $\fg \xto{\ad_{\fg}}\Hom(\fg,\fg)$.
Clearly, $\fg$ and $\ad_{\fg}$ determine one another.
Indeed, using the formalism developed in \S\ref{ss:eight}, 
the rule $\fg \mapsto \ad_{\fg}$ embeds the category of Lie algebras
over $\ZZ$ into the ``homotopy category'' of
module representations over $\ZZ$.

Let $p$ be a prime and let $\fg$ be a finite nilpotent Lie algebra
$\fg$ of class $< p$ over $\ZZ/p^n$.
As we will see in \S\ref{s:cc_duality}, by embedding the category of
such Lie algebras into a category of module representations over $\ZZ/p^n$,
the function $\ask{\dtimes\,}$ (defined on the latter category) extends
taking the class number $\concnt(\exp(\fg))$ of the finite $p$-group $\exp(\fg)$
associated with $\fg$.
Crucially, Knuth duals of adjoint representations of Lie algebras
need not be of the same form.
It is therefore reasonable to enlarge the original ambient category of
Lie algebras.

A natural (but possibly intractable) question suggested by this
point of view asks which numbers of the form $\ask\theta$ (where
$\theta$ is a module representation) are, or are related to, class
numbers $\concnt(G)$ of groups.
We contribute two results towards answering this question.
First, we will see in \S\ref{s:cc_ask2} that $\ask \alpha$ is essentially a class number
for each ``alternating'' module representation $\alpha$, a 
class of module representations far more general than adjoint representations of Lie algebras.
Secondly, we will see that for each module representation
$M\xto\theta\Hom(V,W)$ (subject to some finiteness conditions), the number 
\[
\askm \theta 2 := \frac 1{\card M} \sum_{a\in M}\card{\Ker(a\theta)}^2
\]
is, up to a harmless factor, the class number of a group.
We will discuss the relationship between the numbers $\ask\theta$ and $\askm
\theta 2$ and give several examples of the latter.

\subsection*{Acknowledgement}

The author is grateful to Christopher Voll for comments and
discussions and to the anonymous referee for several helpful and insightful
suggestions.

\subsection*{\textit{Notation}}

Throughout, rings are assumed to be associative, commutative, and
unital.
Unless otherwise specified by subscripts, tensor products are taken over their natural
base rings.
Maps (including matrices) usually act on the right.
For an object $A$, we often let $A$ itself refer to the identity morphism $A \xto{\mathrm{id}_A} A$.
The symbol $\fO$ always denotes a discrete valuation ring (\DVR) and $\fP$
denotes its maximal ideal.
We write $\fO_n = \fO/\fP^n$, $\RF = \fO/\fP$, and $q = \card{\RF}$
whenever $\RF$ is finite.
We let $K$ denote the field of fractions of $\fO$.

\section{Module representations}
\label{s:mrep}

Throughout, let $R$ be a ring.

\subsection{Modules}
\label{ss:modules}

We recall standard terminology and set up notation.
Let $\MOD(R)$ be the category of $R$-modules and $R$-homo\-mor\-phisms
and let $\MODf(R)$ be the full subcategory of $\MOD(R)$ consisting of finitely generated modules.
Let $(\dtimes)^* = \Hom(\dtimes,R)$ denote the usual dual.

The \emph{evaluation map} associated with $R$-modules $V$ and $W$ is
\[
V \xto{\eval V W} \Hom\Bigl(\Hom(V,W),\,W\Bigr),\quad
x \mapsto (\phi \mapsto x\phi).
\]
Recall that $V$ is \emph{reflexive} if $V \xto{\eval V R} V^{**}$ is an
isomorphism.
The following is folklore.

\begin{prop}
  \label{prop:fg_projective}
  Let $P$ and $P'$ be finitely generated projective
  $R$-modules.
  Then:
  \begin{enumerate}
  \item \label{prop:projective_reflexive1}
    $P$ is reflexive.
  \item \label{prop:projective_reflexive2}
    $P^*$ is finitely generated and projective.
  \item \label{prop:projective_reflexive3}
    The natural maps $P^*\otimes P'
    \to \Hom(P,P')$  and $P^* \otimes (P')^* \to (P\otimes P')^*$ are both
    isomorphisms.
\end{enumerate}
\end{prop}
\begin{proof}
  See e.g.~\cite[Ch.\ 5, p.\ 66, Thm]{Jan64}
  for (\ref{prop:projective_reflexive1})--(\ref{prop:projective_reflexive2})
  and \cite[12.14--12.19]{milneCA} for (\ref{prop:projective_reflexive3}).
\end{proof}

Projective modules over local rings are free; see e.g.\ \stacks{0593}.
If $R$ is Noetherian and $V$ and $W$ are finitely generated
$R$-modules, then $\Hom(V,W)$ is finitely generated.

\subsection{Eight categories of module representations}
\label{ss:eight}

\paragraph{Module representations.}
A \emph{module representation} over $R$ is a homomorphism
$M \xto\theta\Hom(V,W)$, where $M$, $V$, and $W$ are $R$-modules.
We occasionally write $M = \Mod\theta$, $V = \Dom\theta$,
and $W = \Cod\theta$.
A module representation $\theta$ can be equivalently described in terms of
the bilinear \emph{multiplication}
$\Dom\theta\times \Mod\theta \xto{\mul\theta} \Cod\theta$ 
defined by $x \mul\theta a = x(a\theta)$ for $x\in \Dom\theta$ and $a\in \Mod\theta$.
The maps $\mul\theta$ are precisely the ``bimaps'' studied by Wilson et al.;
see e.g.\ \cite{Wil13,BW14,Wil15,Wil17} and also related work of First~\cite{Fir15}.
As Wilson points out~\cite[\S 1]{Wil13}, there are numerous natural categories
of bimaps---and hence of module representations.

\paragraph{The categories $\MREP_\tau(R)$.}
Let $\up$ and $\dn$ be distinct symbols and
let $\bar\dtimes$ be the operation which interchanges them.
For a category $\cC$, let $\cC^\up = \cC$ and let $\cC^\dn = \cC^\opp$
be the dual of~$\cC$.
Given a string $\tau = \tau_1\dotsb\tau_n$ over $\{\up,\dn\}$, let
$\cC^\tau = \cC^{\tau_1} \times\dotsb\times\cC^{\tau_n}$ and $\bar\tau
= \bar\tau_1\dotsb\bar\tau_n$.

\begin{defn}
  \label{d:mrep}
Let $\Types =\{ \tau = \tau_1\tau_2\tau_3 : \tau_i\in \{\up,\dn\}\}$.
For each $\tau \in \Types$, we define a category $\MREP_\tau(R)$ 
of module representations and ``$\tau$-morphisms'' over $R$ as
follows:
\begin{enumerate}
\item \label{d:mrep1}
  The objects of $\MREP_\tau(R)$ are the module representations
  $M \xto\theta\Hom(V,W)$ over $R$.
\item \label{d:mrep2}
  A \emph{$\tau$-morphism}
  $\bigl( M \xto\theta \Hom(V,W)\bigr)
  \longrightarrow
  \bigl( \tilde M \xto{\tilde\theta} \Hom(\tilde V,\tilde W)\bigr)$
  is a morphism $(M,V,W) \to (\tilde M,\tilde
  V,\tilde W)$ in $\MOD(R)^\tau$ which ``intertwines $\mul\theta$ and
  $\mul{\tilde\theta}$'' (see below);
  composition is inherited from $\MOD(R)^\tau$.
\end{enumerate}
\end{defn}

Rather than give a general but technical definition
of the intertwining condition, we list each case separately.
Let $M \xto{\theta}\Hom(V,W)$ and $\tilde M \xto{\tilde\theta} \Hom(\tilde V,
\tilde W)$ be module representations over~$R$.

\begin{itemize}
\item An \emph{$\uuu$-morphism} or \emph{homotopy} $\theta
  \to\tilde\theta$ is a triple of homomorphisms $(M \xto{\nu} \tilde M,
  V\xto{\phi} \tilde V, W \xto{\psi} \tilde W)$ such that $(x
  \mul\theta a)\psi = (x\phi) \mul{\tilde\theta} (a\nu)$ for all
  $a\in M$ and $x \in V$.
  Equivalently, the following diagram is required to commute for
  each $a\in M$:
  \[
  \begin{CD}
    {V} @>{a\theta}>> {W}\\
    @V{\phi}VV @VV{\psi}V\\
    {\tilde V} @>{(a\nu)\tilde\theta}>> \tilde W.
  \end{CD}
  \]
\item A \emph{\ddu-morphism} $\theta\to\tilde\theta$ is a triple of
  homomorphisms
  $(\tilde M \xto{\tilde \nu} M, \tilde V \xto{\tilde\phi} V, W \xto\psi
  \tilde W)$ such that
  $\tilde x \mul{\tilde\theta}\tilde a = (\tilde x \tilde \phi
  \mul\theta \tilde a \tilde \nu)\psi$
  for all $\tilde a \in \tilde M$ and $\tilde x \in \tilde V$.
\item An \emph{\udd-morphism} $\theta\to\tilde\theta$ is a triple of
  homomorphisms
  $(M \xto{ \nu} \tilde M, \tilde V \xto{\tilde\phi} V, \tilde W \xto{\tilde\psi} W)$
  such that
  $\tilde x\tilde\phi \mul\theta a = (\tilde x \mul{\tilde\theta} a\nu)\tilde\psi$
  for all $a\in M$ and $\tilde x\in\tilde V$.
\item A \emph{$\dud$-morphism} $\theta \to\tilde\theta$ is a triple of homomorphisms
  $(\tilde M \xto{\tilde\nu} M, V\xto{\phi} \tilde V, \tilde W \xto{\tilde\psi} W)$ such that
  $x\mul\theta \tilde a\tilde\nu  = (x\phi \mul{\tilde\theta}\tilde a)\tilde\psi$
  for all $\tilde a\in \tilde M$ and $x \in V$.
  Equivalently, the following diagram is required to commute:
  \[
  \begin{CD}
    {\tilde M} @>{\tilde\theta}>> {\Hom(\tilde V,\tilde W)}\\
    @V{\tilde\nu}VV @VV{\Hom(\phi,\tilde\psi)}V\\
    {M} @>{\theta}>> {\Hom(V,W)}.
  \end{CD}
  \]
\item
  A $\bar\tau$-morphism $\theta\to\tilde\theta$
  is a $\tau$-morphism $\tilde\tau\to\tau$.
  (Hence, $\MREP_{\bar\tau}(R) = \MREP_\tau(R)^\opp$.)
\end{itemize}

\clearpage

\begin{rem}
  \label{rem:mrep}
  \quad
  \begin{enumerate}
  \item \label{rem:mrep1}
    We write $\MREP(R) := \MREP_{\uuu}(R)$ for the \emph{homotopy
      category} of module representations over $R$.
    Our use of the word homotopy,
    as well as of \emph{isotopy} (=~invertible homotopy), goes back to work of
    Albert~\cite{Alb42}.
    Wilson~\cite[\S 1.2]{Wil13} calls such morphisms homotopisms and isotopisms,
    respectively;
    his homotopism category of $(R,R)$-bimaps is equivalent to our
    $\MREP(R)$.
  \item \label{rem:mrep2} 
    Two module representations $\theta$ and $\tilde\theta$ are
    $\tau$-isomorphic for some $\tau\in \Types$ if and only if they 
    are isotopic.
    We write $$\theta \approx \tilde\theta$$ to signify the
    existence of an isotopy $\theta \to \tilde\theta$.
  \end{enumerate}
\end{rem}

\subsection{Additive structure: direct and collapsed sums}
\label{ss:addition}

Given module representations $M\xto\theta \Hom(V,W)$ and $\tilde M
\xto{\tilde\theta}\Hom(\tilde V,\tilde W)$ over $R$, their \emph{direct sum}
(= categorical biproduct) $\theta\oplus\tilde\theta$ in $\MREP(R)$ is the
module representation
\[
M \oplus \tilde M \xto{\theta\oplus\tilde\theta} \Hom(V\oplus\tilde
V,W\oplus \tilde W),
\quad
(a,\tilde a) \mapsto a\theta \oplus \tilde a\tilde\theta.
\]

For any module $U$, let $U^m \xto{\Sigma_U^m} U$ denote $m$-fold
addition and let $U \xto{\Delta_U^m} U^m$ be the diagonal map;
for $m = 2$, we often drop the superscripts.
Let $\theta$ and $\tilde\theta$ be module representations over $R$.
Suppose that $\Mod\theta=\Mod{\tilde\theta}$. We may then ``collapse'' the
module $\Mod{\theta\oplus\tilde\theta} = \Mod\theta^2$ down to a single copy of
$\Mod\theta$ by considering the module representation
\[
\Mod\theta \xto{\Delta_{\Mod\theta}} \Mod\theta^2 \xlongequal{\phantom{???}}
\Mod{\theta\oplus\tilde\theta} \xto{\theta\oplus\tilde\theta}
\Hom(\Dom{\theta\oplus\tilde\theta},\Cod{\theta\oplus\tilde\theta}).
\]
Similarly, if $\Dom\theta = \Dom{\tilde\theta}$ or $\Cod\theta =
\Cod{\tilde\theta}$, we may collapse $\Dom{\theta\oplus\tilde\theta}$
and $\Cod{\theta\oplus\tilde\theta}$ down to $\Dom\theta$ and
$\Cod\theta$, respectively, by considering the module
representations 
\begin{align*}
  \Mod{\theta\oplus\tilde\theta}&
  \xto{\theta\oplus\tilde\theta}\Hom(\Dom{\theta\oplus\tilde\theta},\Cod{\theta\oplus\tilde\theta})
  \xto{\Hom(\Delta_{\Dom\theta},\Cod{\theta\oplus\tilde\theta})}
  \Hom(\Dom\theta, \Cod{\theta\oplus\tilde\theta}) \text{\quad and} \\
  \Mod{\theta\oplus\tilde\theta}&
  \xto{ \theta\oplus\tilde\theta}
  \Hom(\Dom{\theta\oplus\tilde\theta},\Cod{\theta\oplus\tilde\theta})  
  \xto{\Hom(\Dom{\theta\oplus\tilde\theta},\Sigma_{\Cod\theta})}
  \Hom(\Dom{\theta\oplus\tilde\theta},\Cod\theta),
\end{align*}
respectively.
We refer to such module representations as \emph{collapsed sums}.
We shall be particularly interested in module representations of the
form $\theta^m := \theta \oplus \dotsb\oplus \theta$ ($m$-fold sum).
In this case, we derive module representations
$\coll\theta m := \Delta_{\Mod{\theta}}^m \dtimes \theta^m,$ $\theta^m \dtimes
\Hom(\Dom{\theta}^m,\Sigma^m_{\Cod{\theta}})$, and $\theta^m \dtimes \Hom(\Delta_{\Dom{\theta}}^m,\Cod{\theta}^m)$ 
which we call the \emph{collapsed $m$th powers} of~$\theta$.

\section{Average sizes of kernels and \Ask{} zeta functions---revisited}
\label{s:ask}

We collect and adapt notions and results from \cite{ask}, expressed
in the language developed here. 
Let $R$ be a ring; unless otherwise indicated, all module representations are over~$R$.

\subsection{Average sizes of kernels}

For a module representation $M\xto\theta\Hom(V,W)$ such that $M$ and $V$ are
finite as sets and a non-negative integer $m$, let
\[
\askm\theta m = \frac 1 {\card M} \sum_{a\in M}\card{\Ker(a\theta)}^m;
\]
note that $\askm\theta m$ only depends on the isotopy class of $\theta$.
Further note that if $M/\Ker(\theta) \xto{\bar\theta}\Hom(V,W)$ denotes the
map induced by $\theta$, then
$\askm \theta m = \askm {\bar\theta} m$.
We write $\ask\theta = \askm\theta 1$ for the average size
of the kernel of the elements of $M$ acting as linear maps $V \to W$ via $\theta$.
Even though this was not previously observed in \cite{ask}, the study of the
function $\ask{\dtimes\,}$ actually includes the more general functions
$\askm{\dtimes\,}m$; see Lemma~\ref{lem:askm_coll}.

\subsection{Direct and collapsed sums}

Recall the definitions of $\theta\oplus\tilde\theta$, $\theta^m$, and
$\coll\theta m$ from \S\ref{ss:addition}.

\begin{lemma}[{Cf.\ \cite[Lem.\ 2.2]{ask}}]
  \label{lem:delta_ask}
  Let $M \xto\theta\Hom(V,W)$ and 
  $\tilde M \xto{\tilde\theta}\Hom(\tilde V,\tilde W)$ be module
  representations.
  Suppose that the modules $M$, $\tilde M$, $V$, and $\tilde V$ are
  all finite as sets.
  Then
  $\ask{\theta\oplus\tilde\theta} = \ask\theta \ask{\tilde\theta}$.
  In particular, $\ask{\theta^m} = {\ask\theta}^m$.
\end{lemma}

The average size of the kernel associated with the first collapsing
operation in \S\ref{ss:addition} admits the following description.

\begin{lemma}
  \label{lem:askm_coll}
  Let $M\xto\theta\Hom(V,W)$ and $M\xto{\tilde\theta}\Hom(\tilde
  V,\tilde W)$ be module representations.
  Suppose that the modules $M$, $V$, and $\tilde V$ are finite as sets.
  Then
  \[
  \ask{\Delta_M (\theta\oplus\tilde\theta)} =
  \frac 1{\card M} \sum_{a\in M} \card{\Ker(a\theta)}\dtimes \card{\Ker(a\tilde\theta)}.
  \]
  In particular, $\ask{\coll \theta m} = \askm \theta m$.
\end{lemma}
\begin{proof}
  $\Ker(a (\Delta_M (\theta \oplus \tilde\theta))) = \Ker(a\theta) \oplus
  \Ker(a\tilde\theta)$ for $a\in M$.
\end{proof}

We will encounter another natural source of numbers $\askm{\dtimes\,}2$
in Theorem~\ref{thm:ask_sk} below.

\subsection{Change of scalars for module representations}
\label{ss:ext}

Let $\tilde R \xto\phi R$ be a ring homomorphism.
Given a module representation $\tilde M\xto{\tilde\theta}\Hom(\tilde V,\tilde
W)$ over $\tilde R$,
we define the \emph{change of scalars} $\tilde \theta^\phi$ of $\tilde\theta$ along $\phi$
to be the module representation
$\tilde M\otimes_{\tilde R} R \to \Hom(\tilde V \otimes_{\tilde R} R, \,\tilde
W\otimes_{\tilde R} R)$,
where we regarded $R$ as an $\tilde R$-module via $\phi$;
when the reference to $\phi$ is clear from the context, we simply write
$\tilde\theta^R = \tilde\theta^\phi$.

\subsection{On \Ask{} and \Askm m{} \!zeta functions}
\label{ss:ask_zeta}

Let $\fO$ be a compact \DVR{}; recall the notation from the end of \S\ref{s:intro}.
For $y\in \fO$, write $\fO_y = \fO/y$ and let $(\dtimes)_y =
(\dtimes)\otimes \fO_y\colon \MOD(\fO)\to\MOD(\fO_y)$.
For $n \ge 0$, we write $(\dtimes)_n = (\dtimes)_{\pi^n}$, where $\pi \in
\fP\setminus\fP^2$ is an arbitrary uniformiser.
Let $M\xto\theta\Hom(V,W)$ be a module representation over $\fO$.
Suppose that $M$, $V$, and $W$ are finitely generated.
Then $M_y$, $V_y$, and $W_y$ are finite as sets for each $y\in \fO\setminus\{0\}$.

\begin{defn}[{Cf.\ \cite[Def.\ 1.3]{ask}}]
The \emph{\Askm m{}\! zeta function} of $\theta$ is the generating function
\[
\Zeta^m_\theta(T) = \sum_{n=0}^\infty
\askm{\theta^{\fO_n}} m \,T^n \in \QQ\llb T \rrb.
\]
We write $\Zeta_\theta(T) := \Zeta^1_\theta(T)$ for the \emph{\Ask{} zeta function} of $\theta$.
If $M\subset \Hom(V,W)$ and $\theta_M$ denotes this inclusion map, then we write
$\Zeta^m_M(T) := \Zeta^m_{\theta_M}(T)$ and $\Zeta_M(T) := \Zeta_{\theta_M}(T)$.
\end{defn}

\begin{ex}[{Cf.\ \cite[\S 1]{ask}}]
  $\displaystyle \Zeta^m_{\{0\} \incl \Mat_{d}(\fO)}(T) =
  \sum\limits_{n=0}^\infty q^{dmn} T^n = 1/(1-q^{dm}T)$.
\end{ex}

By Lemma~\ref{lem:askm_coll}, 
$\Zeta^m_\theta(T) = \Zeta_{\coll\theta m}(T)$.
Thus, it might seem that \Askm m\! zeta functions offer
nothing new over the \Ask{} zeta functions studied in~\cite{ask}.
On the contrary, we will see in \S\ref{s:cc_ask2} that \Askm 2 zeta
functions naturally arise in the study of class numbers of finite groups.

\subsection{Integral formalism and rationality}
\label{ss:intrat}

Let $\fO$ and $\theta$ be as in \S\ref{ss:ask_zeta}.
For $y\in \fO$, let $M \xonto{\pi_y} M_y$ denote the projection.
Define
\[
\kersize_\theta\colon M \times \fO \to \RR\cup\{\infty\}, \quad
(a,y) \mapsto \bigl\lvert\Ker((a\pi_y) \theta^{\fO_y})\bigr\rvert.
\]

Let $\mu_U$ denote the Haar measure on a finitely generated $\fO$-module $U$
with $\mu_U(U) = 1$.
Let $\abs{\dtimes}$ denote the absolute value on the field of
fractions $K$ of $\fO$ with $\abs\pi = q^{-1}$
for $\pi\in\fP\setminus\fP^2$, where $q$ denotes the residue field size of $\fO$.
We write $$\zeta^m_\theta(s) := \Zeta_\theta^m(q^{-s})$$ and similarly
$\zeta_\theta(s)$, $\zeta_M(s)$, etc.\ for the Dirichlet series associated
with $\Zeta_\theta^m(T)$.

\begin{thm}[{Cf.\ \cite[Thms~4.5, 4.10]{ask}}]
  \label{thm:int}
  Let $M \xto\theta \Hom(V,W)$ be a module representation over~$\fO$.
  Suppose that $M$, $V$, and $W$ are free of finite ranks.
  Then for all $s \in \CC$ with $\Real(s) > m \dtimes \rank(V)$,
  \[
  (1-q^{-1}) \dtimes \zeta^m_\theta(s) =
  \int_{M \times \fO} \abs{y}^{s-1} \dtimes \kersize_\theta(a,y)^m
  \,\dd\mu_{M\times \fO}(a,y).
  \]
  Moreover, if $\fO$ has characteristic zero, then $\Zeta^m_\theta(T) \in \QQ(T)$.
\end{thm}

\subsection{Constant rank and minimality}
\label{ss:crk}

Let $M \xto\theta\Hom(V,W)$ be a module representation over a \itemph{field}.
The \emph{generic rank} $\genrank(\theta)$ of $\theta$ is $\sup(\rank(a\theta): a\in M)$.
We say that $\theta$ has \emph{constant rank $r$} if $M \not= 0$ and $a\theta$
has rank $r$ for each $a\in M \setminus \{0\}$.
Let $\fO$ be a compact \DVR{} with field of fractions $K$.
Let $M \xto\theta\Hom(V,W)$ be a module representation over $\fO$,
where $M$, $V$, and $W$ are free of finite ranks.
Let $r = \genrank(\theta^K)$ (see \S\ref{ss:ext}), $d = \dim(V\otimes K)$, and
$\ell = \dim(M\otimes K)$.

\begin{defn}[{Cf.\ \cite[\S 6.1]{ask}}]
  \label{d:Kmin}
  We say that $\theta$ is \emph{faithfully $\kersize$-minimal} if 
  $\kersize_\theta(a,y) = \abs{y}^{r-d}$
  for all $a\in M\setminus\fP M$ and $y\in \fO\setminus\{0\}$.
  We say that  $\theta$ is \emph{$\kersize$-minimal} if the induced module representation
  $M/\Ker(\theta) \xto\theta \Hom(V,W)$ is faithfully
  $\kersize$-minimal.
\end{defn}

Using a variation of \cite[Lem.~5.6]{ask}, one can show that
$\theta$ is $\kersize$-minimal if and only $M\theta\subset\Hom(V,W)$ is
$\kersize$-minimal in the sense of \cite[Def.\ 6.3]{ask}. 

\begin{prop}[{Cf.\ \cite[Prop.\ 6.4]{ask} (for $m=1$)}]
  \label{prop:askm_minimal}
  \quad
  \begin{enumerate}
  \item 
  \label{prop:askm_minimal1}
    If $\theta$ is faithfully $\kersize$-minimal, then
  \begin{equation}
    \label{eq:Kmin}
    \Zeta^m_\theta(T) =
    \frac{1-q^{ (d-r)m - \ell} T}
    {
      (1-q^{dm-\ell}T)
      (1-q^{(d-r)m}T)
    }.
  \end{equation}
\item
  \label{prop:askm_minimal2}
  Conversely, if \eqref{eq:Kmin} holds for any $m\ge 1$, then $\theta$ is
  faithfully $\kersize$-minimal. \qedhere 
  \end{enumerate}
\end{prop}
\begin{proof}
  For (\ref{prop:askm_minimal1}),
  combine Theorem~\ref{thm:int} and \cite[Lem.\ 5.8]{ask}.
  Part~(\ref{prop:askm_minimal2}) is analogous to the ``if part'' of
  \cite[Prop.\ 6.4]{ask}.
\end{proof}

The following is a local counterpart of \cite[Prop.\ 6.8]{ask}.

\begin{prop}
  \label{prop:minimal_crk}
  Let $M\xto\theta\Hom(V,W)$ be a module representation over $\fO$.
  Suppose that each of $M$, $V$, and $W$ is free of finite rank.
  Let $r = \genrank(\theta^K)$ and suppose that $\theta^{\RF}$
  has constant rank $r$, where $\RF = \fO/\fP$.
  Then $\theta$ is faithfully $\kersize$-minimal.
\end{prop}
\begin{proof}
  Let $a \in M\setminus \fP M$ and let $\bar a$ denote the image of $a$ in
  $M \otimes \RF$.
  Since $\theta^{\RF}$ has constant rank $r$,
  for each $i=1,\dotsc,r$, 
  some $i\times i$ minor of $\bar
  a\theta^\RF$ is non-zero in $\RF$.
  In particular, $a\theta$ admits a unit $i\times i$ minor for each $i=
  1,\dotsc,r$ whence the $r$ non-zero elementary divisors of $a\theta$ are all
  equal to $1$.
  In particular, $\kersize_\theta(a,y) = \abs y^{r-\dim(V\otimes K)}$ for each
  $y\in \fO\setminus\{0\}$.
\end{proof}

\section{Knuth duality for module representations}
\label{s:duality}

Let $R$ be a ring.
We will now relate categories of the form $\MREP_\tau(R)$ by suitable functors.

\subsection{The functors \MV, \MW, and \VW}
\label{ss:duals}

Recall the notation from \S\ref{ss:modules}.
Any module representation $M \xto{\theta}\Hom(V,W)$ over $R$ gives rise to
associated module representations $\theta^\MV$, $\theta^\MW$, and $\theta^\VW$
defined as follows:
\begin{alignat*}{6}
  V &\xto{\theta^\MV} && \Hom(M,W), \quad\quad&& x &\,&\mapsto\theta \dtimes x
  \eval V W &&= \bigl(a \mapsto x(a\theta)\bigr), \\
  W^*&\xto{\theta^\MW} && \Hom(V,M^*), \quad\quad&& \psi &&\mapsto
  \theta^{\MV} \dtimes \Hom(M,\psi) &&= \bigl(x \mapsto (a \mapsto (x(a\theta))\psi)\bigr),
  \\
  M &\xto{\theta^{\VW}} && \Hom(W^*,V^*), \quad\quad&& a &&\mapsto (a\theta)^* &&=
  \bigl(\psi \mapsto (x\mapsto (x(a\theta))\psi)\bigr).
\end{alignat*}

These operations generalise Knuth's action on ``cubical arrays''~\cite[\S 4.1]{Knu65b};
see \S\ref{ss:matrices} and cf.~\cite{Lie81}.
The operations $\MV$, $\MW$, and $\VW$ are related:
for $a \in M$, $x \in V$, and $\psi \in W^*$, 
\begin{equation}
  \label{eq:triple}
  (x (a\theta)) \psi = (a (x\theta^\MV))\psi =  a (x (\psi \theta^{\MW})) =
  x (\psi (a \theta^{\VW}));
\end{equation}
moreover, $\theta^\VW = \theta^{\MV\MW\MV}$ 
and $\theta^\MW = \theta^{\MV\VW\MV}$ (Lemma~\ref{lem:conj_dual}).

Given a triple of $R$-module homomorphisms $(\nu,\phi,\psi)$, we define
\[
(\nu,\phi,\psi)^\MV = (\phi,\nu,\psi), \quad\quad
(\nu,\phi,\psi)^\MW = (\psi^*,\phi,\nu^*), \quad\quad
(\nu,\phi,\psi)^\VW = (\nu,\psi^*,\phi^*).
\]

\begin{prop}
  \label{prop:functors}
  Let $\star \in \{\MV,\MW,\VW\}$.
  \begin{enumerate}
  \item \label{prop:functors1}
    The operation $\star$ defines functors $$\MREP(R)
    \xtofrom[(\dtimes\,)^\star]{(\dtimes\,)^\star} \MREP_{\tau_\star}(R)$$
    and
    $$\MREP_{\ddd}(R) \xtofrom[(\dtimes\,)^\star]{(\dtimes\,)^\star} \MREP_{\bar\tau_\star}(R),$$
    where $\tau_\star$ is given in the following table:
    \begin{center}
      \begin{tabular}{c|ccc}
        $\star$ & \MV & \MW & \VW\\
        \hline
        $\tau_\star$ & \uuu & \dud & \udd.
      \end{tabular}
    \end{center}
  \item \label{prop:functors2}
    By (\ref{prop:functors1}), $(\dtimes)^{\star\star}$
    defines an endofunctor of each of $\MREP(R)$, $\MREP_\ddd(R)$,
    $\MREP_{\tau_\star}(R)$, and $\MREP_{\bar\tau_\star}(R)$.
    In each case, we obtain a natural transformation
    $\mathrm{Id} \xto{\eta^\star} (\dtimes)^{\star\star}$ with $\theta$-component 
    $(\nu_\theta,\phi_\theta,\psi_\theta)$
    (where $\theta$ is a module representation over $R$)
    as indicated in the following table:
    \begin{center}
      \begin{tabular}{c|ccc}
        $\star$ &$\nu_\theta$&$\phi_\theta$&$\psi_\theta$ \\
        \hline
        \MV & $\Mod{\theta}$ & $\Dom{\theta}$ & $\Cod{\theta}$ \\
        \MW & $\eval {\Mod{\theta}} R$ & $\Dom{\theta}$ & $\eval {\Cod{\theta}} R$ \\
        \VW & $\Mod{\theta}$ & $\eval {\Dom\theta} R$ & $\eval {\Cod\theta} R$.
      \end{tabular}
    \end{center}
  \end{enumerate}
\end{prop}
\begin{proof}
  This is easy but tedious.
  We only spell out a proof that $\eta^\MW_\theta$ is a
  homotopy $\theta\to\theta^{\MW\MW}$ for each $M \xto\theta\Hom(V,W)$ in
  $\MREP(R)$.
  Let $x\in V$, $a\in M$, and $\psi \in W^*$.
  Then
  $\psi ((x \mul\theta a) \eval W R)
  = \psi((x(a\theta))\eval W R)
  = (x(a\theta))\psi$.
  Next, two applications of \eqref{eq:triple} yield
  \begin{align*}
    \psi(x\mul{\theta^{\MW\MW}} (a\eval M R))
     = \psi(x((a\eval M R)\theta^{\MW\MW}))
       \underset{\eqref{eq:triple}}= (x (\psi\theta^\MW))(a\eval M R)
      = a(x(\psi\theta^\MW)) \underset{\eqref{eq:triple}}= (x(a\theta))\psi.
  \end{align*}
  Therefore, $(x\mul\theta a)\eval W R = x \mul{\theta^{\MW\MW}}(a\eval M R)$
  and it follows that $\eta_\theta^\MW$ is a homotopy.
\end{proof}

\begin{rem}
  $\MV$ and $\VW$ correspond to 
  Knuth's~\cite{Knu65b} ``dual'' and ``transpose'' of semifields (see
  e.g.\ \cite[\S 4]{LP14} and cf.\ \S\ref{ss:matrices} below).
  Furthermore, $\MV$ and $\MW$ correspond
  to Wilson's ``swap''~\cite[\S 2]{Wil17} and ``Knuth-Liebler
  shuffle''~\cite[\S 3]{Wil15} operations for bimaps.
\end{rem}

\subsection{Reflexivity and truncated \DVR{}s}
\label{ss:reflexive}

We now seek to find reasonably rich subcategories of $\MREP_\tau(R)$ which turn
the functors \MV,~\MW, and \VW{} into equivalences.
As a first step towards achieving this goal, in this section,
we  consider full subcategories $\cC$ of $\MOD(R)$ which satisfy the following conditions:
\begin{itemize}[leftmargin=3.5em]
\item[\RFL] $\cC$ consists of finitely generated reflexive modules.
\item[\STB] $\cC$ is stable under taking duals in $\MOD(R)$.
\end{itemize}

We note that reflexivity assumptions are frequently imposed in the
study of dualities between module categories; see e.g.\ \cite[\S 23]{AF92}.

\begin{ex}
The full subcategory of $\MOD(R)$
consisting of finitely generated projective modules satisfies \RFL{}--\STB{}
by Proposition~\ref{prop:fg_projective}.
\end{ex}

A ring is \emph{quasi-Frobenius (\QF)} if it is Noetherian
and if it is an injective module over itself;
see e.g.\ \cite[Ch.\ 6]{Lam99}. 
It is known that \QF{} rings are Artinian; see \cite[Thm~15.1]{Lam99}.
\begin{prop}
  If $R$ is a \QF{} ring, then $\MODf(R)$ satisfies \RFL{}--\STB{}.
\end{prop}
\begin{proof}
  Duals of finitely generated modules over a Noetherian ring are finitely generated.
  By \cite[Thm~15.11(2)]{Lam99}, finitely generated modules over \QF{} rings
  are reflexive.
\end{proof}

\begin{defn}
  A \emph{truncated \DVR{}} is a ring which is isomorphic to a ring of the
  form $\fO_n = \fO/\fP^n$ for some \DVR{}~$\fO$ with maximal ideal $\fP$ and
  some $n \ge 0$.
\end{defn}

The following intrinsic characterisation of truncated \DVR{}s is due to Hungerford~\cite{Hun68}.

\begin{thm}
  \label{thm:hungerford}
  Let $R$ be a non-zero principal ideal ring (not necessarily a domain).
  Then the following are equivalent:
  \begin{enumerate}
  \item \label{thm:hungerford1}
    $R$ has a unique prime ideal.
  \item \label{thm:hungerford2}
    $R$ is local and Artinian.
  \item \label{thm:hungerford3}
    $R$ is a truncated \DVR{}.
  \end{enumerate}
\end{thm}
\begin{proof}
  The implication (\ref{thm:hungerford1})$\to$(\ref{thm:hungerford2}) follows
  from \cite[Thm~2.14]{Eis95}.
  The maximal ideal of a local Artinian ring is nilpotent by \stacks{00J8}.
  By \stacks{00JA}, such a ring contains a unique prime ideal.
  Thus, (\ref{thm:hungerford2}) implies (\ref{thm:hungerford1}).
  By \cite[\S 2]{Hun68}, a local principal ideal ring with a unique and
  nilpotent prime ideal is a quotient of a local \PID{} which is not a field.
  Since the latter rings are precisely the \DVR{}s (see
  \cite[Thm~11.2]{Mat86}), this proves 
  that (\ref{thm:hungerford1}) and (\ref{thm:hungerford2}) imply
  (\ref{thm:hungerford3}).
  Finally, (\ref{thm:hungerford3}) obviously implies (\ref{thm:hungerford2}).
\end{proof}

\begin{lemma}
  \label{lem:tdvr}
  Let $R$ be a truncated \DVR{}.
  Then:
  \begin{enumerate}
  \item \label{lem:tdvr1}
    $R$ is a \QF{} ring.
  \item \label{lem:tdvr2}
    If $U$ is a finitely generated $R$-module,
    then $U \approx U^*$ (non-canonically).
    \end{enumerate}
\end{lemma}
\begin{proof}
  \quad
  \begin{enumerate}
  \item
    A non-zero truncated \DVR{} is local and Artinian (by
    Theorem~\ref{thm:hungerford}), and has a simple socle (viz.\ the minimal
    non-zero ideal). 
    By \cite[Thm~15.27]{Lam99}, such a ring is \QF{}.
  \item 
  For $\fa \normal R$, we identify $(R/\fa)^* = \Ann_R(\fa)$
  via $\phi \mapsto (1+\fa)\phi$. 
  Let $R = \fO_n = \fO/\fP^n$, where $\fO$ is a \DVR{} with
  maximal ideal $\fP$.
  Choose $\pi \in \fP\setminus\fP^2$.
  Since taking dual modules commutes with finite direct sums,
  by the structure theory of finitely generated modules over \PID{}s,
  it suffices to consider the case that $U = \fO_e$ for $e \le n$.
  Using the identification
  $U^* = \Ann_{\fO_n}(\fP^e/\fP^n) = \fP^{n-e}/\fP^n$,
  the claim follows since multiplication by $\pi^{n-e}$ induces an isomorphism
  $U \approx \fP^{n-e}/\fP^n$.
  \qedhere
  \end{enumerate}
\end{proof}

\subsection{Equivalences induced by \MV, \MW, and \VW{} and an $\sym_3$-action}
\label{ss:S3}

Let $\cC$ be a full subcategory of $\MOD(R)$ which satisfies \RFL{}--\STB{}
from \S\ref{ss:reflexive}.
For $\tau \in \Types$ (see \S\ref{ss:eight}), let $\MREP_\tau(\cC)$ be
the full subcategory of $\MREP_\tau(R)$ consisting of module representations
$M\to\Hom(V,W)$ with $M$, $V$, and $W$ in $\cC$; write $\MREP(\cC) =
\MREP_\uuu(\cC)$.

\begin{prop}
  \label{prop:equiv}
  Let $\star \in \{\MV,\MW,\VW\}$.
  \begin{enumerate}
  \item \label{prop:equiv0}
    The functors from
    Proposition~\ref{prop:functors}(\ref{prop:functors1})
    restrict to equivalences
    $\MREP(\cC) \xtofrom[(\dtimes\,)^\star]{(\dtimes\,)^\star} \MREP_{\tau_\star}(\cC)$
    and $\MREP_{\ddd}(\cC) \xtofrom[(\dtimes\,)^\star]{(\dtimes\,)^\star} \MREP_{\bar\tau_\star}(\cC)$.
    \textup{(By definition, these are the \emph{Knuth dualities} from the title of \S\ref{s:duality}.)}
  \item \label{prop:equiv1}
    The natural transformations $\eta^\star$ in
    Proposition~\ref{prop:functors}(\ref{prop:functors2}) restrict
    to natural isomorphisms of endofunctors of $\MREP(\cC)$, $\MREP_\ddd(\cC)$,
    $\MREP_{\tau_\star}(\cC)$, and $\MREP_{\bar\tau_\star}(\cC)$.
  \end{enumerate}
\end{prop}
\begin{proof}
  By \STB{}, we obtain the functors in (\ref{prop:equiv0}) as
  restrictions of those in Proposition~\ref{prop:functors}(\ref{prop:functors1}).
  Clearly, a morphism in $\MREP_\tau(\cC)$ is an isomorphism if and only if
  it is an isomorphism in $\MOD(R)^\tau$. 
  Part~(\ref{prop:equiv1}) now follows from \RFL{} and this completes the
  proof of~(\ref{prop:equiv0}).
\end{proof}

By Remark~\ref{rem:mrep}(\ref{rem:mrep2}), the isomorphism classes in
$\MREP_\tau(R)$ (and similarly also in $\MREP_\tau(\cC)$) do not depend on
$\tau$ in the sense that two module representations over $R$ are isomorphic
in $\MREP_\tau(R)$ for \itemph{some} $\tau \in \Types$ if and only if they are
isomorphic in $\MREP_\tau(R)$ for \itemph{each} $\tau \in \Types$.
By Proposition~\ref{prop:functors}(\ref{prop:functors1}), each $\star\in
\{\MV,\MW,\VW\}$ induces a map between isotopy classes
$\sfrac{\MREP(\cC)}{\!\approx} \xto{\phantom{ok}} \sfrac{\MREP(\cC)}{\!\approx}$.
By Proposition~\ref{prop:equiv}, these maps are involutions.
Their interactions are recorded in the following straightforward lemma.

\begin{lemma}
  \label{lem:conj_dual}
  Let $M \xto\theta\Hom(V,W)$ be a module representation over an arbitrary ring.
  Then:
  \begin{enumerate}
  \item \label{lem:conj_dual1}
    $\theta^{\MV\MW\MV} = \theta^\VW$
    and $(\eval M R, W^*,V^*)$ is a homotopy $\theta^\VW \to
    \theta^{\MW\MV\MW}$.
  \item \label{lem:conj_dual2}
    $\theta^{\MV\VW\MV} = \theta^\MW$
    and 
    $(W^*, \eval V R,M^*)$ is a homotopy $\theta^\MW \to
    \theta^{\VW\MV\VW}$.
  \item \label{lem:conj_dual3}
    $(\eval V R, \eval M R, \eval W R)$ is a homotopy $\theta^\MV \to
    \theta^{\MW\VW\MW}$ and a homotopy
    $\theta^\MV \to \theta^{\VW\MW\VW}$.
    \qed
  \end{enumerate}
\end{lemma}

\paragraph{Action.}
As before, let $\cC$ be a full subcategory of $\MOD(R)$ which satisfies
\RFL{}--\STB{}.
We conclude that by letting the transpositions $(1,2)$, $(1,3)$, and
$(2,3)$ in $\sym_3$ act as $\MV$, $\MW$,
and $\VW$, respectively, we obtain an action of $\sym_3$ on
$\sfrac{\MREP(\cC)}{\!\approx}$.
This action permutes the modules $\Mod\theta$, $\Dom\theta$, and $\Cod\theta$
associated with a module representation $M\xto\theta\Hom(V,W)$ up to taking
duals.
To keep track of the latter, it is convenient to replace $\sym_3$ by the
isomorphic copy
\begin{equation}
  \label{eq:signed_S3}
  \Pi
  := \bigl\langle (1,2), (1,-3), (2,-3) \bigr\rangle
  \le \{ \pm 1 \} \wr \sym_3
\end{equation}
 as a \textit{signed} permutation group;
see Table~\ref{tab:S3} for the resulting $\Pi$-action on $\sfrac{\MREP(\cC)}{\!\approx}$.

\begin{table}[h]
  \centering
  \begin{tabular}{cccccrl}
    $\sigma$& elt of $\Pi$ &$\Mod\sigma$ & $\Dom\sigma$ & $\Cod\sigma$
    & rule for $\sigma$\\
    \hline
    $\theta$ & $\mathrm{id} = (\,)$ & $M$ & $V$ & $W$& $a\theta =$ &$ x \mapsto x (a\theta)$ \\
    $\theta^\MV$ &$(1,2)$& $V$ & $M$ & $W$ & $x\theta^\MV =$ & $a \mapsto x (a\theta)$ \\ 
    $\theta^\MW$ &$(1,-3)$& $W^*$ & $V$ & $M^*$ &  $\psi \theta^\MW=$& $x \mapsto \left(a \mapsto (x (a\theta)) \psi \right)$\\
    $\theta^{\MV\MW}$ & $(1,2,-3)$& $W^*$ & $M$ & $V^*$ &  $\psi \theta^{\circ\bullet} = $& $a \mapsto \left( x \mapsto (x (a\theta))\psi \right)$\\
    $\theta^{\MW\MV}$ &$ (1,-3,2)$& $V$ & $W^*$ & $M^*$ & $x \theta^{\bullet\circ} = $ & $\left(\psi \mapsto \left(a \mapsto (x(a\theta))\psi\right) \right)$                                              
    \\
    $\theta^\VW = \theta^{\MV\MW\MV}\approx\theta^{\MW\MV\MW}$ & $(2,-3)$& $M$ & $W^*$ & $V^*$ &
    $a \theta^\VW = $ & $(a\theta)^*$
  \end{tabular}
  \caption{The $\Pi$-orbit of $M\xto\theta\Hom(V,W)$ in $\sfrac{\MREP(\cC)}{\!\approx}$}
  \label{tab:S3}
\end{table}

\begin{rem}[Knuth duality for projective modules]
  Let $M\xto\theta\Hom(V,W)$ be a module representation over a ring.
  Suppose that $M$, $V$, and $W$ are all finitely generated projective.
  Then $\theta$ corresponds naturally to an element of $\Hom(M,\Hom(V,W))
  \approx M^* \otimes V^* \otimes W$ and the three involutions above
  simply permute the tensor factors (up to isomorphisms coming from
  reflexivity); this is essentially Liebler's coordinate-free
  interpretation~\cite{Lie81} of Knuth's $\sym_3$-action on ``cubical arrays''.
\end{rem}

\subsection{Matrices}
\label{ss:matrices}

\begin{quote}
  \textit{``It is my experience that proofs involving matrices can be shortened
  by 50\% if one throws the matrices out.''}

\quad --- Emil Artin~\cite[p.\ 14]{Art88}
\end{quote}

For module representations involving free modules of finite rank, the
preceding duality functors take simple explicit forms.
Let $\zz = (z_1,\dotsc,z_\ell)$ consist of algebraically independent
indeterminates over $R$.
Identify $\Hom(R^d,R^e) = \Mat_{d\times e}(R)$.
Each $A(\zz) \in \Mat_{d\times e}(R[z_1,\dotsc,z_\ell])$ whose entries are
linear forms gives rise to a module representation
\[
R^\ell \xto{A(\dtimes\,)} \Mat_{d\times e}(R), \quad \ww \mapsto A(\ww).
\]

Conversely, 
let $M\xto\theta\Hom(V,W)$ be a module representation, where
$M$, $V$, and $W$ are free $R$-modules of finite ranks $\ell$,
$d$, and $e$, respectively.
Choose (ordered) bases $\cA = (a_1,\dotsc,a_\ell)$, $\cX = (x_1,\dotsc,x_d)$, and $\cY =
(y_1,\dotsc,y_e)$ of $M$, $V$, and $W$, respectively.
For $1\le h \le \ell$, $1\le i \le d$, and $1\le j \le e$, 
define $c_{hij}(\theta;\cA,\cX,\cY) \in R$ by
\[
x_i (a_h\theta) = \sum_{j=1}^e c_{hij}(\theta;\cA,\cX,\cY) \dtimes y_j.
\]
For $1\le h \le \ell$, let
$M_h(\theta;\cA,\cX,\cY) = [ c_{hij}(\theta;\cA,\cX,\cY) ]_{\substack{1\le i \le d\\1 \le
  j\le e}} \in \Mat_{d\times e}(R)$;
hence, $M_h(\theta;\cA,\cX,\cY)$ is the matrix of $a_h \theta$ with
respect to the bases $\cX$ and $\cY$ of $V$ and $W$, respectively.
Finally,
let
$$M(\theta;\cA,\cX,\cY) := \sum\limits_{h=1}^\ell z_h
M_h(\theta;\cA,\cX,\cY) \in \Mat_{d\times e}(R[\zz]).$$

\begin{lemma}
  Let $R^\ell \xto\nu M$, $R^d \xto\phi V$, and $R^e \xto\psi W$ be the
  isomorphisms corresponding to the bases $\cA$, $\cX$, and $\cY$,
  respectively.
  Then $M(\theta;\cA,\cX,\cY)(\dtimes) \xto{(\nu,\phi,\psi)} \theta$
  is an isotopy. \qed
\end{lemma}
\vspace*{0.5em}

Let $\cA^*$, $\cX^*$, and $\cY^*$ denote the dual bases of $\cA$, $\cX$, and
$\cY$, respectively.
Note that
\begin{equation}
  \label{eq:triple_coeff}
  c_{hij}(\theta;\cA,\cX,\cY) = \bigl(x_i(a_h\theta)\bigr)y_j^*
\end{equation}
for all $h,i,j$ as above.
We now describe the effects of $\MV$, $\MW$, $\VW$ in terms of coordinates.

\begin{prop}
  \label{prop:matrix_duality}
  \quad
  \begin{enumerate}
  \item \label{prop:matrix_duality1}
    $c_{hij}(\theta;\cA,\cX,\cY)
    = c_{ihj}(\theta^\MV;\cX,\cA,\cY)
    = c_{jih}(\theta^\MW;\cY^*,\cX,\cA^*)
    = c_{hji}(\theta^\VW;\cA,\cY^*,\cX^*)$.
  \item \label{prop:matrix_duality2}
    $M(\theta^{\VW};\cA,\cY^*,\cX^*) = M(\theta;\cA,\cX,\cY)^\top$.
  \item \label{prop:matrix_duality3}
    $M(\theta^{\MV\MW};\cY^*,\cA,\cX^*) = M(\theta^{\MW};\cY^*,\cX,\cA^*)^\top$.
  \item \label{prop:matrix_duality4}
    $M(\theta^{\MW\MV};\cX,\cY^*,\cA^*) = M(\theta^\MV;\cX,\cA,\cY)^\top$.
  \end{enumerate}
\end{prop}
\begin{proof}
  We only show that $c_{jih}(\theta^\MW;\cY^*,\cX,\cA^*) =
  c_{hij}(\theta;\cA,\cX,\cY)$ and leave the remaining (equally simple) statements to the
  reader.
  By combining \eqref{eq:triple_coeff} and \eqref{eq:triple}, we see that
  \[
  c_{jih}(\theta^\MW;\cY^*,\cX,\cA^*) = 
  \bigl(x_i (y_j^* \theta^\MW)\bigr)a_h^{**}
  = a_h \bigl(x_i (y_j^* \theta^\MW)\bigr)
  = (x_i (a_h\theta)) y_j^* = c_{hij}(\theta;\cA,\cX,\cY). \qedhere
  \]
\end{proof}

Hence, for finitely generated free modules $M$, $V$, and $W$,
the $\sym_3$-action on the orbit of a module representation $M \to \Hom(V,W)$
permutes the indices $(h,i,j)$ from above. 
It is this very concrete form of the duality from \S\ref{ss:S3} which, to the 
author's knowledge, was introduced by Knuth~\cite{Knu65b} in the context of
semifields defined by structure constants.

\begin{rem}[Commutator matrices]
  \label{rem:comm_mat}
  Let $\fg$ be a Lie algebra over~$R$.
  Let $\fz$ and $[\fg,\fg]$ denote the centre and derived subalgebra of $\fg$,
  respectively.
  Suppose that $\fg/\fz$ and $[\fg,\fg]$ are both free $R$-modules with finite
  bases $\cA$ and $\cY$, say.
  Let $\fg/\fz \xto{\theta} \Hom(\fg/\fz,[\fg,\fg])$ be the module representation 
  whose associated multiplication $\fg/\fz \times \fg/\fz \xto{\mul\theta}
  [\fg,\fg]$ is induced by the Lie bracket of $\fg$.
  O'Brien and Voll~\cite[Def.\ 2.1]{O'BV15} attached two commutator matrices 
  $A(\mathbf X)$ and $B(\mathbf Y)$ to $\fg$ and the chosen bases.
  These matrices coincide with our $M(\theta;\cA,\cA,\cY)$ and
  $M(\theta^\MW;\cY^*,\cA,\cA^*)$, respectively.
  Note that, by anti-commutativity of Lie brackets,
  $M(\theta^\MV;\cA,\cA,\cY) = -M(\theta;\cA,\cA,\cY)$.
\end{rem}

\subsection{Duals of collapsed sums}
\label{ss:collapsed}

In the following, for $R$-modules $U$ and $V$, we identify
$(U \oplus V)^* = U^* \oplus V^*$
and thus $(U^*)^m = (U^m)^*$.
Consequently, for module representations $\theta$ and $\tilde\theta$ over $R$
and $\star\in\{\MV,\MW,\VW\}$,
we may identify $(\theta\oplus\tilde\theta)^\star = \theta^\star \oplus\tilde\theta^\star$.
Moreover, we will now see that collapsing direct sums as in \S\ref{ss:addition} is 
compatible with the duality operations from
\S\ref{s:duality} in the expected way: the collapsed module is moved
according to the $\sym_3$-action from \S\ref{ss:S3}.
We only spell out the case $\star=\MW$, the others being similar.

\begin{prop}
  \label{prop:coll_MW}
  Let $M \xto\theta \Hom(V,W)$ and $\tilde M\xto{\tilde\theta}\Hom(\tilde V,\tilde
  W)$ be module representations over $R$.
  \begin{enumerate}
  \item
    \label{prop:coll_MW1}
    If $M = \tilde M$, then
    $(\Delta_M \dtimes (\theta\oplus\tilde\theta))^{\MW} = (\theta^\MW \oplus
    \tilde\theta^\MW) \dtimes \Hom(V\oplus\tilde V,\Sigma_{M^*})$.
  \item
    \label{prop:coll_MW2}
    If $W = \tilde W$, then
    $((\theta\oplus\tilde\theta) \dtimes \Hom(V\oplus\tilde V,\Sigma_W))^\MW = \Delta_{W^*}
    \dtimes (\theta^{\MW} \oplus \tilde\theta^\MW)$.
  \item
    \label{prop:coll_MW3}
    If $V = \tilde V$, then
    $((\theta\oplus\tilde\theta) \dtimes\Hom(\Delta_V,W\oplus\tilde W))^\MW
    = (\theta^\MW \oplus\tilde\theta^\MW) \dtimes \Hom(\Delta_V,M^*
    \oplus\tilde M^*)$.
  \end{enumerate}
\end{prop}

Part~(\ref{prop:coll_MW3}) of Proposition~\ref{prop:coll_MW} simply
reflects the fact that the element $(1,-3) \in \Pi$ (see \eqref{eq:signed_S3})
corresponding to $\MW$ fixes the point $2$.
Parts (\ref{prop:coll_MW1})--(\ref{prop:coll_MW2}) follow from the next lemma
(applied in the special case that $\psi = \mathrm{id}_W$ or $\tilde\nu =
\mathrm{id}_M$, respectively) and the two identities $\Sigma_U^* =
\Delta_{U^*}$ and $\Delta_U^* = \Sigma_{U^*}$ which are valid for any
$R$-module~$U$.

\begin{lemma}
  \label{lem:bullet_func}
    Let $M \xto\theta\Hom(V,W)$ be a module representation over $R$.
    Let $\tilde M \xto{\tilde\nu} M$ and $W \xto{\psi} \tilde W$ be
    module homomorphisms.
    Then
    \[
    (\theta \dtimes \Hom(V,\psi))^\MW \dtimes \Hom(V,\tilde\nu^*)
    = \psi^* \dtimes (\tilde\nu\dtimes \theta)^\MW.
    \]
\end{lemma}
\begin{proof}
  First, let $M_i \xto{\theta_i} \Hom(V,W_i)$ ($i=1,2$) be module
  representations over $R$ and let $\theta_1 \xto{(\tilde\nu,V,\psi)} \theta_2$  be a homotopy.
  Then the following diagram commutes:
  \[
  \begin{CD}
    W_2^* @>{\theta_2^{\MW}}>> \Hom(V,M_2^*)\\
    @V{\psi^*}VV @VV{\Hom(V,\tilde\nu^*)}V\\
    W_1^* @>{\theta_1^\MW}>> \Hom(V,M_1^*);
  \end{CD}
  \]
  indeed, this is equivalent to $(\tilde\nu,V,\psi)^\MW$
  being a $\dud$-morphism $\theta_1^\MW \to \theta_2^\MW$ and it is
  thus a consequence of Proposition~\ref{prop:functors}(\ref{prop:functors1}).
  The claim follows by defining $\theta_1$ to be the composite
  $\tilde M \xto{\tilde\nu} M \xto\theta\Hom(V,W)$ and $\theta_2$ to be
  $M \xto\theta\Hom(V,W)\xto{\Hom(V,\psi)}\Hom(V,\tilde W)$.
\end{proof}

\begin{cor}
  \label{cor:bullet_func}
  Let the notation be as in Lemma~\ref{lem:bullet_func}.
  Then
  $$(\tilde\nu \dtimes \theta \dtimes \Hom(V,\psi))^{\MW} =
 \psi^* \dtimes \theta^{\MW} \dtimes  \Hom(V, \tilde\nu^*).$$
\end{cor}
\begin{proof}
  Apply Lemma~\ref{lem:bullet_func} twice.
\end{proof}

\section{Effects of Knuth duality on average sizes of kernels}
\label{s:S3_effects}

In this section, we explore the effects of the $\sym_3$-action from
\S\ref{ss:S3} on average sizes of kernels and \Ask{} zeta functions
associated with module representations.

\subsection{Module bundles}
\label{ss:bundles}

Let $R$ be a ring.
An \emph{$R$-module bundle} (frequently referred to as a
\emph{bundle} below) over a base set $Y$ consists of a set $X$ and
a partition $X = \coprod_{y\in Y} X_y$ into 
$R$-modules $X_y$.
A \emph{morphism} $X\xto\chi\tilde X$ of bundles over $Y$ is a map which
induces module homomorphisms $X_y \to \tilde X_y$ for $y\in Y$;
we let $\approx_Y$ signify the existence of an isomorphism of
bundles over~$Y$.
The \emph{dual} of a bundle $X$ over $Y$ is $X^* = \coprod_{y\in Y} X_y^*$.

Clearly, the category of $R$-module bundles over $Y$ is equivalent to the
product category $\MOD(R)^Y$.
Our justification for introducing new terminology is that a single set can
occasionally be regarded as a bundle over different base sets.
For instance, given $R$-modules $U$ and $V$, we may naturally identify $U\times V =
\coprod_U V = \coprod_V U$ and thus regard $U\times V$ as a bundle over $U$ and over $V$.

Let $M \xto\theta\Hom(V,W)$ be a module representation over $R$.
We may regard $\theta$ as a morphism $V\times M\xto{\theta} W\times M,\,
(x,a)\mapsto (x(a\theta),a)$ of bundles over $M$.
The \emph{kernel} and \emph{cokernel} of this morphism are given by
\begin{align*}
  \KER(\theta) := \coprod_{a\in M} \Ker(a\theta) \text{\quad and\quad} 
  \COKER(\theta) := \coprod_{a\in M} \Coker(a\theta),
\end{align*}
respectively.
Define $\comm(\theta) = \{ (x,a)\in V\times M : x \mul\theta a = 0\}$.
We identify $\comm(\theta) = \KER(\theta)$,
where we regard $\comm(\theta)$ as a bundle over $M$ via the projection
$\comm(\theta)\to M$.
Using the projection $\comm(\theta)\to V$, 
we then also regard $\KER(\theta)$ as a bundle over~$V$.

\begin{prop}
  \label{prop:mrep_bundle}
  There are canonical isomorphisms of $R$-module bundles:
  
  \begin{enumerate}
  \item \label{prop:mrep_bundle1}
    $\KER(\theta^\MV) \approx_V \KER(\theta)$.
  \item \label{prop:mrep_bundle2}
    $\KER(\theta^\VW) \approx_M \COKER(\theta)^*$.
  \item \label{prop:mrep_bundle3}
    $\KER(\theta^\MW) \approx_V \COKER(\theta^\MV)^*$.
  \end{enumerate}
\end{prop}
\begin{proof}
  For (\ref{prop:mrep_bundle1}), simply note that
  $\comm(\theta)\to\comm(\theta^\MV), \, (x,a)\mapsto (a,x)$ is an isomorphism
  over $V$.
  For (\ref{prop:mrep_bundle2})--(\ref{prop:mrep_bundle3}),
  by \eqref{eq:triple}, 
  \begin{align*}
    \comm(\theta^\VW)
    & = \{ (\psi,a) \in\ W^* \times M : \psi (a\theta) = 0 \}
    \\
    & = \{ (\psi,a) \in\ W^* \times M : \Img(a\theta) \subset \Ker(\psi)\}
    \\
    & \approx_M \coprod_{a\in M}(W/\Img(a\theta))^*
    \\
    & = \COKER(\theta)^*
  \end{align*}
  and, analogously,
  \begin{align*}
    \comm(\theta^\MW)
    & = \{ (x,\psi)\in V\times W^* : x(\psi\theta^\MW) = 0\} \\
    & = \{ (x,\psi) \in  V\times W^* : \Img(x\theta^\MV) \subset
      \Ker(\psi)\} \\&\approx_V \COKER(\theta^\MV)^*.
      \qedhere
  \end{align*}
\end{proof}

\subsection{Average sizes of kernels of duals}
\label{ss:ask_duality}

Let $R$ be a \textit{finite}
truncated \DVR{} (see \S\ref{ss:reflexive}). 
We now show that the $\sym_3$-action from \S\ref{ss:S3} all but
preserves average sizes of kernels over $R$.
For the formal statement, let $M \xto \theta\Hom(V,W)$ be a module representation over $R$,
where each of $M$, $V$, and $W$ is finite.
Recall that $\ask\theta$ (see \S\ref{s:ask}) only depends on the isotopy class
(see \S\ref{ss:eight}) of $\theta$.

\begin{thm}
  \label{thm:ask_duality}
  \quad
  \begin{enumerate}
  \item \label{thm:ask_duality1}
    $ \ask{\theta^\MV} \!= \frac{\card M}{\card V} \ask\theta$.
  \item \label{thm:ask_duality2}
     $\ask{\theta^\VW} \!= \frac{\card W}{\card V} \ask\theta$.
  \item \label{thm:ask_duality3}
    $\ask{\theta^{\MW}} \!= \ask\theta$.
  \end{enumerate}
\end{thm}

\begin{proof}
  First note that by the definition of $\KER(\theta)$ (see \S\ref{ss:bundles}),
  \begin{equation}
    \label{eq:ask_comm}
    \ask \theta = \frac{\card{\KER(\theta)}}{\card M}.
  \end{equation}
  \begin{enumerate}
  \item
    By Proposition~\ref{prop:mrep_bundle}(\ref{prop:mrep_bundle1}),
    \[
    \ask{\theta^{\MV}}
    = \frac{\card{\KER(\theta^\MV)}}{\card{V}}
    =  \frac{\card{M}}{\card V} \dtimes \frac{\card{\KER(\theta)}}{\card{M}}
    = \frac{\card M}{\card  V} \ask\theta.
    \]
  \item
    By Lemma~\ref{lem:tdvr}(\ref{lem:tdvr2}),
    ${\COKER(\theta)^*} \approx_M {\COKER(\theta)}$ (non-canonically) whence
    \begin{align*}
      \card{\COKER(\theta)^*}
      & = \sum_{a\in M} \card{W/\Img(a\theta)} 
      = \frac{\card W}{\card V} \sum_{a\in M}
        \card{\Ker(a\theta)} 
      = \frac{\card M \dtimes \card W}{\card V}\ask\theta.
    \end{align*}

    Proposition~\ref{prop:mrep_bundle}(\ref{prop:mrep_bundle2})
    now implies that
    \[
    \ask{\theta^\VW}
    = \frac{\card{\KER(\theta^\VW)}}{\card M}
    = \frac{\card{\COKER(\theta)^*}}{\card M}
    = \frac {\card W}{\card V}\ask\theta.
    \]
  \item
    Combine parts (\ref{thm:ask_duality1})--(\ref{thm:ask_duality2})
    and Lemma~\ref{lem:conj_dual}(\ref{lem:conj_dual2}).
    \qedhere
  \end{enumerate}
\end{proof}

\begin{rem}
  \label{rem:LW}
  For another proof of
  Theorem~\ref{thm:ask_duality}(\ref{thm:ask_duality1}), 
  \[
  \ask\theta
  = \sum_{x \in V}\card{x(M\theta)}^{-1}
  = \sum_{x \in V}\card{\Img(x \theta^\MV)}^{-1}
  = \frac 1 {\card M} \sum_{x \in V} \card{\Ker(x\theta^\MV)}
  = \frac {\card V}{\card M} \ask {\theta^\MV},
  \]
  where the first equality is \cite[Lem.\ 2.1]{ask}, a generalisation
  of an observation due to Linial and Weitz~\cite{LW00}.
  In particular, the duality operation $\MV$ is the source
  of the ``numerical duality'' for average sizes of kernels 
  in \cite[\S\S 4--6]{ask}; see also \S\ref{ss:ask_zeta_duality}.
  In the same spirit, in the special case that $V$ and $W$ are free,
  Theorem~\ref{thm:ask_duality}(\ref{thm:ask_duality2}) follows from
  \cite[Lem.~2.4]{ask}.
\end{rem}

Is there an analogue of Theorem~\ref{thm:ask_duality} for the functions
$\askm{\dtimes\,}m$?
Using the tools developed here, we obtain the following
generalisation of Theorem~\ref{thm:ask_duality}(\ref{thm:ask_duality3});
similar identities hold for $\askm{\theta^\MV}m$ and $\askm{\theta^\VW}m$. 

\begin{cor}
    $\askm {\theta^{\MW}} m = \ask{\theta^m \dtimes \Hom(V^m,\Sigma_W^m)}$.
\end{cor}
\begin{proof}
  By combining Lemma~\ref{lem:askm_coll},
  Proposition~\ref{prop:coll_MW}, 
  and Theorem~\ref{thm:ask_duality}(\ref{thm:ask_duality3}),
  we obtain
  \[
  \askm{\theta^\MW} m =
  \ask{\coll{(\theta^\MW)}m} =
  \ask{(\coll{(\theta^\MW)}m)^\MW} =
  \ask{\theta^m \dtimes \Hom(V^m,\Sigma^m_W)}. \qedhere
  \]
\end{proof}

\begin{rem}[Beyond truncated \DVR{}s.]
  The following generalises \cite[Pf of Prop.~3.4(ii)]{ask}.
  Let $R = R_1\times R_2$ be a direct product of rings.
  It is well-known that, as an abelian group, every $R$-module $U$
  decomposes as $U = U_1 \times U_2$, where $U_i = U \otimes R_i$ is an
  $R_i$-module via the projection $R \onto R_i$;
  indeed, this decomposition furnishes an equivalence between $\MOD(R)$ and
  $\MOD(R_1)\times \MOD(R_2)$.
  Let $M \xto\theta\Hom(V,W)$ be a module representation over $R$.
  Decompose $\theta = (\theta_1,\theta_2)$, where $\theta_i$ is the evident map
  $M_i\xto{\theta_i}\Hom(V,W)_i \xlongequal{\phantom{z}} \Hom(V_i,W_i)$;
  that is (using the notation from \S\ref{ss:ext})
  $\theta_i = \theta^{R \onto R_i}$.
  Hence, for $a = (a_1,a_2) \in M = M_1\times M_2$, 
  $\Ker(a\theta) = \Ker(a_1\theta_1) \times \Ker(a_2\theta_2)$. 
  In particular, if $M$ and $V$ are finite as sets,
  then $\ask\theta = \ask{\theta_1}\ask{\theta_2}$.
  We conclude that Theorem~\ref{thm:ask_duality} remains valid verbatim if $R$
  is allowed to be a (finite) direct product of truncated \DVR{}s.
  Note, in particular, that such rings include finite quotients of
  Dedekind domains.
\end{rem}

\subsection{On \Ask{} zeta functions of duals}
\label{ss:ask_zeta_duality}

In this section, let $\fO$ be a compact \DVR{}.
Let $q$ be the residue field size of $\fO$.
The following is an immediate consequence of Theorem~\ref{thm:ask_duality}.

\begin{cor}
  \label{cor:ask_zeta_duality}
  Let $M\xto\theta\Hom(V,W)$ be a module representation over $\fO$.
  Suppose that $M$, $V$, and $W$ are free of finite ranks $\ell$, $d$, and
  $e$, respectively.
  Then
  \[
  \Zeta_\theta(T) = \Zeta_{\theta^\MV}(q^{d-\ell}T) =
  \Zeta_{\theta^\VW}(q^{d-e} T) = \Zeta_{\theta^\MW}(T). 
  \pushQED{\qed}
  \qedhere
  \popQED
  \]
\end{cor}
\vspace*{1em}

\noindent
Henceforth, let the notation be as in Corollary~\ref{cor:ask_zeta_duality}.

\paragraph{Computations.}
In order to compute $\Zeta_\theta(T)$ via Theorem~\ref{thm:int}, after
choosing bases, we represent $\theta$ 
by a matrix of linear forms over $\fO$ as in \S\ref{ss:matrices}.
We then obtain a formula for $\kersize_\theta(a,y)$ in terms of norms
of minors of said matrix; see \cite[Cor.~4.9]{ask}.
We thus derive an explicit (but usually rather unwieldy) 
expression for $\Zeta_\theta(T)$ as a $\fP$-adic integral.
Corollary~\ref{cor:ask_zeta_duality} now allows us to replace
$\theta$ by any of its six ``$\sym_3$-conjugates'' in the sense of
\S\ref{ss:S3}.
In fact, for practical computations, it suffices to consider the three
module representations $\theta$, $\theta^\MV$, and $\theta^\MW$
only---this follows from
Proposition~\ref{prop:matrix_duality}(\ref{prop:matrix_duality2})--(\ref{prop:matrix_duality4})
and the fact that ideals of minors of a matrix remain invariant under transposition.
The point is that the complexity of the integrals associated with $\theta$,
$\theta^\MV$, and $\theta^\MW$ can vary drastically.

As an illustration, let $\theta$ denote the identity on $\Hom(V,W)$, where $V$
and $W$ are free of finite ranks $d$ and $e$ as above. 
Then the procedure described in the preceding paragraph expresses
$\kersize_\theta(a,y)$ in terms of the ideals of minors of the generic
$d\times e$ matrix.
While these ideals have been studied extensively (see e.g.\ \cite{BV88}),
we will see in Example~\ref{ex:another_Mdxe} that
$\Zeta_\theta(T)$ can be much more easily determined by passing to a
suitable dual.

\paragraph{$\kersize$-minimality and $\orbsize$-maximality.}
Two crucial notions introduced and explored in \cite{ask} were
$\kersize$-minimality (see \S\ref{ss:crk}) and $\orbsize$-maximality
of modules of matrices. 
In \cite[\S 6]{ask}, the former condition was found to be closely related to
constant rank spaces. The latter condition naturally
appeared in \cite[\S 5]{ask} in the computation of \Ask{} zeta
functions associated with classical Lie algebras. 
Dually to the definition of $\kersize_\theta$ in \S\ref{ss:intrat},
define
\[
\imgsize_\theta\colon M \times \fO \to \RR\cup\{\infty\}, \quad
(a,y) \mapsto \bigl\lvert\Img((a\pi_y) \theta^{\fO_y})\bigr\rvert.
\]
By the first isomorphism theorem,
\begin{equation}
  \label{eq:IK}
  \kersize_\theta(a,y) \imgsize_\theta(a,y) = \abs{y}^{-d}
\end{equation}
for all $a\in M$ and $y\in \fO\setminus\{0\}$.
Note that the ``additive orbit'' $x (M\theta) = \{x (a\theta):a\in M\}$ of $x\in V$
as in \cite[\S 2.2]{ask} is precisely the image of $x\theta^\MV$.
Using \eqref{eq:IK},
we now recognise the formal analogy between $\kersize$-minimality and
$\orbsize$-maximality indicated in \cite[\S 6]{ask} as an instance of the
duality operations explored here.
Namely, $M\theta$ is $\orbsize$-maximal in the sense of \cite[\S 5.1]{ask}
if and only if $\theta^\MV$ is $\kersize$-minimal in the sense of
Definition~\ref{d:Kmin}.
This provides a more conceptual interpretation of some
computations of \Ask{} zeta functions in \cite{ask}. 

\begin{ex}[{Another look at $\Mat_{d\times e}(\fO)$}]
  \label{ex:another_Mdxe}
  Let $V =\fO^d$, $W = \fO^e$,
  $M = \Hom(V,W) = \Mat_{d\times e}(\fO)$,
  $\iota = \mathrm{id}_{M}$, and $\theta = \iota^\MV$.
  The proof of \cite[Lem.\ 5.2]{ask} establishes 
  $\genrank(\theta) = e$ and that $\theta^{\RF}$ has constant rank $e$.
  Propositions~\ref{prop:askm_minimal}--\ref{prop:minimal_crk}
  therefore imply that
  \[
  \Zeta_\theta(T) =
  \frac{1-q^{de-d-e} T}
  {
    (1-q^{de-d} T)
    (1-q^{de-e} T)
  }.
\]
  Using Corollary~\ref{cor:ask_zeta_duality},
  we therefore recover \cite[Prop.\ 1.5]{ask} in the form
  \[
  \Zeta_{\Mat_{d\times e}(\fO)}(T) = \Zeta_\iota(T) = \Zeta_\theta(q^{d-de}T) = 
  \frac{1-q^{-e}T}{(1-T)(1-q^{d-e}T)}.
  \]
\end{ex}

\begin{rem}
  \label{rem:so_sym}
  Let $\theta$ be one of the module representations
  $\So_d(\fO) \incl \Mat_d(\fO)$ (for $d\ge 2$, $\mathrm{char}(K) \not= 2$) or
  $\Sym_d(\fO)\incl \Mat_d(\fO)$ from \cite[\S 5]{ask}.
  Similarly to Example~\ref{ex:another_Mdxe},
  from our present point of view, the computation of $\Zeta_\theta(T)$ in
  \cite{ask} establishes and exploits the $\kersize$-minimality of
  $\theta^\MV$.
  Curiously, in both cases, $\theta^\MW \approx \theta^{\MV\VW}$ too
  is $\kersize$-minimal; we will see in \S\ref{ss:MW_min} that this
  is not a general phenomenon.
\end{rem}

\begin{ex}[Duals of band matrices]
  Let $r \ge 1$ and let $\theta_r$ be the inclusion of
  \[
    \left\{
      \begin{bmatrix}
        x_1 \\
        x_2 & \ddots \\
        \vdots & \ddots & \ddots \\
        x_r & \ddots & \ddots & x_1 \\
        & \ddots & \ddots & x_2 \\
        & & \ddots & \vdots \\
        & & & x_r
      \end{bmatrix}
      : x_1,\dotsc,x_r \in \fO
    \right\}
  \]
  into $\Mat_{(2r-1)\times r}(\fO)$.
  By \cite[Ex.\ 6.6]{ask}, $\theta$ is $\kersize$-minimal and
  $\Zeta_{\theta_r}(T) = (1-q^{-1}T)/(1-q^{r-1}T)^2$.
  Passing to $\theta_r^\MW$ offers nothing new since, as is easily
  verified (e.g.\ using Proposition~\ref{prop:matrix_duality}), $\theta_r
  \approx \theta_r^\MW$.
  On the other hand, $\theta_r^\MV$ is isotopic to the inclusion, $\sigma_r$ say,
  of
  \[
  \left\{
    \begin{bmatrix}
      z_1 & z_2 & z_3 & \hdots & z_r \\
      z_2 & z_3 & \iddots &\iddots & z_{r+1} \\
      z_3 & \iddots &\iddots &\iddots&\vdots\\
      \vdots &\iddots &\iddots &\iddots &\vdots\\
      z_r & z_{r+1} &\hdots &\hdots & z_{2r-1}
    \end{bmatrix}
    : z_1,\dotsc,z_{2r-1} \in \fO
  \right\}
  \]
  into $\Mat_r(\fO)$.
  Hence, by Corollary~\ref{cor:ask_zeta_duality},
  \[
  \Zeta_{\sigma_r}(T)
  = \Zeta_{\theta_r}(q^{1-r}T) = \frac{1-q^{-r}T}{(1-T)^2}
  = \Zeta_{\Mat_r(\fO)}(T).
  \]
\end{ex}

It is not surprising that $\kersize$-minimality of 
the $\sym_3$-conjugates $\theta,\theta^\MV,\dotsc$ of $\theta$
should be linked as in Example~\ref{ex:another_Mdxe} or
Remark~\ref{rem:so_sym}.
For instance, by combining Corollary~\ref{cor:ask_zeta_duality} and
Proposition~\ref{prop:askm_minimal}, we easily obtain results of the
following form.

\begin{lemma}
  Let the notation be as in Corollary~\ref{cor:ask_zeta_duality}.
  Let $\theta$ be faithfully $\kersize$-minimal.
  \begin{enumerate}
  \item  If $\{\genrank(\theta),\ell\} = \{\genrank(\theta^\MW),e\}$,
    then $\theta^\MW$ is faithfully $\kersize$-minimal.
  \item
    If $\genrank(\theta^\MV) = \genrank(\theta) - d + \ell$
    and $d\in\{ \genrank(\theta),\ell\}$, then $\theta^\MV$ is
    faithfully $\kersize$-minimal. \qed
  \end{enumerate}
\end{lemma}

\subsection{Example: $\kersize$-minimality of $\theta^\MW$ and
  Westwick's constant rank spaces}
\label{ss:MW_min}

Using a construction due to Westwick~\cite{Wes90}, we now define
module representations $\theta_2,\theta_3,\dotsc$ such that
$\theta_r^\MW$ is $\kersize$-minimal while neither $\theta_r$ nor
$\theta_r^\MV$ is.
Let $\fO$ be a compact \DVR{}.

\paragraph{Westwick's family of constant rank spaces.}
Westwick~\cite{Wes90} defined a matrix of linear forms $H_r(s)\in
\Mat_{(rs+1)\times(rs+s-1)}(\ZZ[X_0,\dotsc,X_s])$ over $\ZZ$;
let $\gamma_{rs}$ be the associated (injective!) module representation 
$\ZZ^{s+1} \to \Mat_{(rs+1)\times(rs+s-1)}(\ZZ)$.
He showed that $\gamma_{rs}^{\CC}$ has constant rank $rs$.
It follows that for each $(r,s)$, if the residue characteristic of $\fO$ is
sufficiently large (depending only on $(r,s)$), then $\gamma_{rs}^{\fO}$ is
$\kersize$-minimal (see \cite[Prop.\ 6.8]{ask}); we will tacitly
assume this in the following.
Up to relabelling variables,
\[
H_r(2) = 
\begin{bmatrix}
  \alpha_1 Y & \beta_1 Z &  \\
  X & \alpha_2 Y & \beta_2 Z & \\
  & X & \alpha_3 Y & \beta_3 Z & \\
  & & \ddots & \ddots & \ddots \\
  &&&\ddots& \alpha_{2r}Y & \beta_{2r} Z\\
  &&&& X & \alpha_{2r+1} Y\\
\end{bmatrix} \in \Mat_{2r+1}(\ZZ[X,Y,Z]),
\]
where $\alpha_{r+1} = 0$, $\beta_r = -1$, and $\alpha_i = \beta_j = 1$ otherwise.

\paragraph{Defining $a_r(\XX)$.}
Let $r\ge 1$.
Write $\XX = (X_0,\dotsc,X_{2r})$ and
let

{\small\[
a_r(\XX) = \begin{bmatrix}
  0 & X_0 & X_1 \\
  X_0 & X_1 & X_2 \\
  \vdots & \vdots & \vdots\\
  X_{r-3} & X_{r-2} & X_{r-1} \\
  X_{r-2} & X_{r-1} & \textcolor{blue}{-X_r \phantom{X}}\\
  X_{r-1} & \textcolor{blue}{0} & X_{r+1} \\
  X_r & X_{r+1} & X_{r+2} \\
  \vdots & \vdots & \vdots\\
  X_{2r-2}& X_{2r-1} & X_{2r} \\
  X_{2r-1} & X_{2r} & 0
\end{bmatrix}
\in \Mat_{(2r+1)\times 3}(\ZZ[\XX]).
\]}

For example, the first three of these matrices are
\[
\begin{bmatrix}
  0 & X_0 & -X_1\\
  X_0 & 0 & X_2 \\
  X_1 & X_2 & 0
\end{bmatrix},
\quad
\begin{bmatrix}
  0 & X_{0} & X_{1} \\
  X_{0} & X_{1} & - X_{2} \\
  X_{1} & 0 & X_{3} \\
  X_{2} & X_{3} & X_{4} \\
  X_{3} & X_{4} & 0
\end{bmatrix}, 
\quad
\begin{bmatrix}
  0 & X_{0} & X_{1} \\
  X_{0} & X_{1} & X_{2} \\
  X_{1} & X_{2} & -X_{3} \\
  X_{2} & 0 & X_{4} \\
  X_{3} & X_{4} & X_{5} \\
  X_{4} & X_{5} & X_{6} \\
  X_{5} & X_{6} & 0
\end{bmatrix}.
\]

\paragraph{Minimality of $\theta_r$.}
  Let $\fO^{2r+1} \xto{\theta_r}\Mat_{(2r+1)\times 3}(\fO)$ be the module
  representation defined by $(x_0,\dotsc,x_{2r})\theta_r = a_r(x_0,\dotsc,x_{2r})$.
  Suppose that $r\ge 2$.
  Then $\genrank(\theta_r^K) = 3$ but
  $\theta_r$ is not $\kersize$-minimal.
  To see that, consider $w = (w_0,\dotsc,w_{2r}) \in \fO^{2r+1}$ with $w_i =
  \delta_{ir}$ (``Kronecker delta'') and note that $w\theta_r$ has rank $2$ over $K$.
  Similarly, one finds that $\theta_r^\MV$ is not $\kersize$-minimal.
  However,
  by performing explicit calculations with matrices as in
  Proposition~\ref{prop:matrix_duality}, 
  one verifies that $\theta_r^\MW \approx \gamma_{r2}^\fO$.
  In particular, Proposition~\ref{prop:askm_minimal},
  Corollary~\ref{cor:ask_zeta_duality},
  and Westwick's result imply that
\[
\Zeta_{\theta_r}(T) = \frac{1-q^{-2}T}{(1-q^{2(r-1)}T)(1-qT)}.
\]

The relevance of this example is that,
using only the techniques developed in \cite{ask}
(i.e.~without the $\MW$-operation), the computation of $\Zeta_{\theta_r}(T)$ would
have been infeasible.

\section{Conjugacy classes of nilpotent groups I: duality}
\label{s:cc_duality}

We discuss consequences of \S\ref{s:S3_effects} for the enumeration of
conjugacy classes of finite groups.

\subsection{Two types of conjugacy class zeta functions}
\label{ss:cc_zeta}

Recall that $\concnt(H)$ denotes the class number of a finite group $H$.

\paragraph{Linear groups.}
Let $\fO$ be a compact \DVR{}.
For a linear group $G \le \GL_d(\fO)$,
du~Sautoy's~\cite{dS05}
\emph{conjugacy class zeta function} is the generating function
\[
\Zeta^\cc_G(T) = \sum_{n=0}^\infty \concnt(G_n)T^n \in \QQ\llb T\rrb,
\]
where $G_n$ denotes the image of $G$ in $\GL_d(\fO_n)$.
(He considered the Dirichlet series $\zeta_{G}^{\cc}(s) := \Zeta_G^{\cc}(q^{-s})$.)
As his main result~\cite[Thm~1.2]{dS05}, he showed that $\Zeta_G^{\cc}(T) \in
\QQ(T)$ for $\fO = \ZZ_p$.
If $\bar G \le \GL_d(\fO)$ denotes the closure of $G$,
then $\Zeta^{\cc}_G(T)$ = $\Zeta^{\cc}_{\bar G}(T)$ (cf.~\cite[Lem.~8.5]{ask}).
We may therefore regard these zeta functions as invariants of $\fO$-linear
profinite groups.

\paragraph{Group schemes.}
Let $R$ be a ring which contains only finitely many ideals $\fa
\normal R$ with $\card{R/\fa} = n$ for each finite $n$.
Let $\sG$ be a group scheme of finite type over~$R$.
Define the \emph{conjugacy class zeta function} of $\sG$ to be the 
(formal) Dirichlet series
\[
\zeta_{\sG}^{\cc}(s) = \sum\limits_\fa \concnt(\sG(R/\fa)) \dtimes \card{R/\fa}^{-s},
\]
where the summation extends over those $\fa\normal R$ with $\card{R/\fa}<\infty$.
Under suitable assumptions, this series defines an analytic function.
Berman et al.~\cite{BDOP13} and Lins~\cite{Lin18a,Lin18b}
studied such conjugacy class zeta functions attached to particular types of
group schemes.

\paragraph{Global setup over number fields.}
Let $k$ be a number field with ring of integers~$\fo$.
Let~$\Places_k$ denote the set of non-Archimedean places of $k$.
For $v\in \Places_k$, let $\fo_v$ be the valuation ring of the $v$-adic
completion $k_v$ of $k$ and let $\fp_v$~be the maximal ideal of $\fo_v$.

Given a linear algebraic group $\GG \le \GL_d\otimes k$ over $k$,
let $\sG$ denote its schematic closure in $\GL_d\otimes \fo$.
Since $\GG$ is smooth over $k$, for almost all $v\in \Places_k$, the natural
map $\sG(\fo_v) \to \sG(\fo_v/\fp_v^n)$ is onto for each $n\ge 1$.
We conclude that for almost all $v\in \Places_k$,
the two conjugacy class zeta functions $\zeta^\cc_{\sG(\fo_v)}(s)$ and
$\zeta^\cc_{\sG\otimes\fo_v}(s)$ coincide.
In particular, up to excluding finitely many places, we may interpret the
conjugacy class zeta functions associated with unipotent algebraic groups over
$k$ in \cite[\S 8.5]{ask} in either of these two ways.

\subsection{Knuth duality for class numbers and conjugacy class zeta functions}
\label{ss:cc_zeta_duality}

\paragraph{Conjugacy class and \Ask{} zeta functions.}
In the following, we freely use the notation for number fields from
\S\ref{ss:cc_zeta}; also recall the notation for changing scalars of a module
representation from \S\ref{ss:ext}.

In the presence of sufficiently good Lie theories,
\Ask{} zeta functions enumerate orbits of groups; this was one
of the author's main motivations for introducing \Ask{} zeta functions in
the first place.
In particular, conjugacy class zeta functions associated with suitable
nilpotent groups \itemph{are} \Ask{} zeta functions:

\begin{prop}[{\cite[Cor.\ 8.19]{ask}}]
  \label{prop:cc_zeta_as_ask}
  Let $k$ be a number field with ring of integers $\fo$.
  Let $\GG \le \GL_d\otimes k$ be a unipotent linear algebraic group
  over $k$ with associated canonical  $\fo$-form~$\sG$.
  Let $\bm\fg \subset \Gl_d(k)$ be the Lie $k$-algebra of $\GG$.
  Let $\fg = \bm\fg \cap \Gl_d(\fo)$.
  Then for almost all $v\in\Places_k$,
  we have $\Zeta_{\sG(\fo_v)}^{\cc}(T) = \Zeta_{\ad_{\fg}^{\fo_v}}(T)$.
\end{prop}

\paragraph{Duality for conjugacy class zeta functions.}
Change of scalars for module representations involving finitely generated
projective modules commutes with taking Knuth duals.
Corollary~\ref{cor:ask_zeta_duality} thus implies the following.

\begin{cor}
  \label{cor:cc_ask_zeta_MW}
  With notation as in Proposition~\ref{prop:cc_zeta_as_ask}, for almost
  all $v\in\Places_k$, we have
  \[\Zeta_{\sG(\fo_v)}^{\cc}(T) = \Zeta_{(\ad_{\fg}^\MW)^{\fo_v}}(T).
    \pushQED{\qed}\qedhere\popQED
  \]
\end{cor}

\begin{rem}
  By anti-commutativity of Lie brackets, $\ad_{\fg}^\MV = -\ad_{\fg}$ 
  whence the consequence $\Zeta_{\sG(\fo_v)}^{\cc}(T) = \Zeta_{(\ad_{\fg}^\MV)^{\fo_v}}(T)$ 
  of Corollary~\ref{cor:ask_zeta_duality} is obvious.
\end{rem}

Stasinski and Voll~\cite[\S 2.1]{SV14} studied ``unipotent group
schemes'' (particular types of integral forms of unipotent algebraic groups)
defined in terms of nilpotent Lie lattices. 
Lins~\cite{Lin18a,Lin18b} introduced and studied bivariate conjugacy class and
representation zeta functions attached to such group schemes.
As explained in \cite[\S 1.2]{Lin18a},
both of her bivariate zeta functions naturally specialise to 
the univariate conjugacy class zeta functions associated with group schemes
in \S\ref{ss:cc_zeta}.
For a second proof of Corollary~\ref{cor:cc_ask_zeta_MW},
we may combine \S\ref{ss:cc_zeta}, Lins's integral formalism \cite[\S
4.2]{Lin18a}, and Remark~\ref{rem:comm_mat}; see \cite[Rem.~4.10]{Lin18a}.
We note that Lins's work relies on a duality for class numbers
(and, in the same way, conjugacy class zeta functions) discovered by O'Brien
and Voll~\cite{O'BV15}; see the next paragraph.

\paragraph{Class numbers and the duality of O'Brien \& Voll.}
Let $R$ be a finite truncated \DVR{}.
Let $\fg$ be a finite nilpotent Lie $R$-algebra of class $c$ such that
$c! \in R^\times$.
Let $G = \exp(\fg)$ be the finite $p$-group corresponding to $\fg$
under the Lazard correspondence; see \cite[Ch.~10]{Khu98}.
A variation of \cite[Prop.\ 8.17]{ask} establishes the following.

\begin{prop}
  \label{prop:ask_concnt}
  $\concnt(G) = \ask{\ad_\fg} = \ask{\ad_\fg^\MW}$. \qed
\end{prop}
\begin{proof}
  By Theorem~\ref{thm:ask_duality}, it suffices to prove the first identity.
  We may identify $G = \fg$ as sets.
  The group multiplication in $G$ is defined in terms of the Lie bracket
  of $\fg$ by means of the Hausdorff series.
  Conversely, we may recover sums and Lie brackets in $\fg$ from the group
  multiplication and group commutators in $G$ via the inverse Hausdorff
  series; see e.g.\ \cite[\S 10.2]{Khu98} for details.
  By inspection of the Hausdorff series and its inverses, one finds that
  two elements of the group $G$
  commute if and only if they commute as elements of the Lie algebra
  $\fg$.
  In particular, as noted in \cite[\S 3.1]{O'BV15},
  for $g\in G = \fg$, the group centraliser $\Cent_G(g)$
  coincides with the Lie centraliser $\mathfrak{c}_{\fg}(g)$.
  Since the latter centraliser coincides with the kernel of
  $\fg\xto{\ad_{\fg}(g)} \fg$,
  using the orbit-counting lemma, we obtain
  \[
    \concnt(G) = \frac 1{\card G} \sum_{g\in G}\card{\Cent_G(g)} = \frac
    1{\card \fg}\sum_{g\in \fg}{\card{\mathfrak c_\fg(g)}} =
    \ask{\ad_{\fg}}. \qedhere 
  \]
\end{proof}

If $\fg/\fz$ and $[\fg,\fg]$ are free as $R$-modules, then
Proposition~\ref{prop:ask_concnt} asserts that $\concnt(G)$ can be
understood in terms of the rank loci of either one of the two
commutator matrices
$A(\mathbf X)$ and $B(\mathbf Y)$ in the work of O'Brien and Voll
\cite{O'BV15}; see Remark~\ref{rem:comm_mat}.
In fact, as we will now explain, Proposition~\ref{prop:ask_concnt}
is an equivalent reformulation of the duality they discovered.
First, as in Remark~\ref{rem:comm_mat},  the adjoint representation of $\fg$
gives rise to a module representation $\fg/\fz \xto{\theta}
\Hom(\fg/\fz,[\fg,\fg])$;
note that $\ask{\ad_{\fg}} = \ask{\theta}\dtimes \card{\fz}$.
By \cite[Thm~A]{O'BV15},
\begin{equation}
  \label{eq:O'BV_concnt_S}
  \concnt(G) = \card{\mathcal S(G)} \dtimes \frac{\card{\fz}}{\card{[\fg,\fg]}}.
\end{equation}
It is easy to see that we may identify the set $\mathcal S(G)$ in
\cite{O'BV15} with our $\comm(\theta^\MW)$ from \S\ref{ss:bundles}.
Thus, using \eqref{eq:ask_comm}, equation \eqref{eq:O'BV_concnt_S} is
equivalent to
$\concnt(G) = \ask{\theta^\MW} \dtimes \card{\fz} = \ask{\ad_\fg^\MW}$.

In the same spirit, we may interpret the double counting argument leading to
the two expressions for $\concnt(G)$ in \cite[Thms~A--B]{O'BV15} as a proof
of $\concnt(G) = \ask{\ad_\fg^\MW} = \ask{\ad_\fg^{\MW\MV}}$;
this is simply another version of Proposition \ref{prop:ask_concnt} since,
by repeated application of Theorem~\ref{thm:ask_duality}, 
\[
\ask{\ad_\fg^{\MW\MV}} =
\ask{(\ad_\fg^\MV)^{\MV\MW\MV}} =
\ask{-\ad_\fg^{\VW}} = \ask{\ad_\fg}.
\]

In summary, we may thus regard Theorem~\ref{thm:ask_duality}
as a simultaneous generalisation of the ``Linial \& Weitz duality''
(part~(\ref{thm:ask_duality1}))
underpinning large parts of \cite{ask} as well as of the ``O'Brien \&
Voll duality'' (part~(\ref{thm:ask_duality3})) in Proposition~\ref{prop:ask_concnt}. 

\begin{rem}
  The key theoretical ingredients of \cite{O'BV15} are the Lazard
  correspondence and the Kirillov orbit method.
  While our proof of Proposition~\ref{prop:ask_concnt}
  was solely based on elementary linear algebra (via
  Theorem~\ref{thm:ask_duality}),
  as explained in \cite{O'BV15}, the Kirillov orbit method
  provides a group-theoretic interpretation of the identity $\concnt(G) =
  \ask{\ad_{\fg}^\MW}$: the right-hand side enumerates irreducible characters
  of $G$.
\end{rem}

For later use, we note that
Proposition~\ref{prop:ask_concnt} provides us with the following analogue
of Proposition~\ref{prop:cc_zeta_as_ask} for conjugacy class zeta functions
associated with group schemes.
For any ring $R$, let $\CALG(R)$ denote the category of associative,
commutative, and unital $R$-algebras.
Let $\GRP$ be the category of groups.

\begin{cor}
  \label{cor:cc_zeta_gpscheme}
  Let $\fO$ be a compact \DVR{} with residue field characteristic $p$.
  Let $\fg$ be a nilpotent Lie $\fO$-algebra of class $<p$ whose underlying
  $\fO$-module is free of finite rank.
  Then we obtain a group scheme $\CALG(\fO) \xto{\exp(\fg\otimes (\dtimes\,))}
  \GRP$
  and
  $\zeta_{\exp(\fg\otimes (\dtimes\,))}^\cc(s) = \zeta_{\ad_{\fg}}(s)$. \qed
\end{cor}
\begin{proof}
  The fact that $\exp(\fg\otimes (\dtimes\,))$ is a group scheme over $\fO$ is
  a minor variation of \cite[\S 2.1.2]{SV14}.
  The second claim follows from Proposition~\ref{prop:ask_concnt}.
\end{proof}

\section{Conjugacy classes of nilpotent groups II: alternation}
\label{s:cc_ask2}

We have seen that conjugacy class zeta functions associated
with certain group schemes are instances of (Dirichlet series attached to)
\Ask{} zeta functions (Corollary~\ref{cor:cc_zeta_gpscheme}).
Conversely, it is then a natural task to characterise those \Ask{} zeta
functions which arise as conjugacy class zeta functions (say, up to shifts
$s\to s+n$ or other simple transformations).
Such a characterisation seems out of reach.
Instead, in this section, we will see that, subject to mild assumptions,
\begin{itemize}
\item
  ``alternating'' module representations and
\item
  \Askm 2\! zeta functions associated with arbitrary module representations
\end{itemize}
always arise as conjugacy class zeta functions.
Throughout, let $R$ be a ring.

\subsection{Class numbers and alternating module representations}
\label{ss:alternating_gps}

A module representation $M \xto\theta \Hom(V,W)$ over $R$ is
\emph{alternating} if $M = V$ and $a\mul\theta a = 0$ for all $a \in M$.
Recall the definitions of $\CALG(R)$ and $\GRP$ from the end of \S\ref{ss:cc_zeta_duality}.
For $A\in \OBJ(\CALG(R))$, 
let $(\dtimes)_A$ denote the base change functor $\MOD(R) \to \MOD(A)$.

Let $M \xto\alpha\Hom(M,W)$ be an alternating module representation over $R$.
Define a functor $\CALG(R)\xto{\sG_\alpha}\GRP$ by endowing, for each $A\in\OBJ(\CALG(R))$,
the set $\sG_\alpha(A) := M_A \times W_A$ with the multiplication
\begin{align*}
  (a,y) \dtimes (a', y') & = \left(a + a', y + y' +
                               ({a}\mul\alpha{a'})
                           \right) 
\end{align*}
for $a,a'\in M_A$ and $y,y'\in W_A$;
define the action of $\sG_\alpha$ on homomorphisms to be the evident one.
Note that $\sG_\alpha(A)$ is a central extension of $W_A$ by $M_A$ and is thus
nilpotent of class at most $2$. 
For a group $H$ and $x,y\in H$, we write $[x,y] = [x,y]_H = x^{-1}y^{-1} xy$.
Clearly,
\begin{equation}
  \label{eq:Galpha_comm}
  (a,y)^{-1}
  = (-a,-y) \quad\text{and}\quad
  [(a,y), (a',y')]_{\sG_\alpha(A)}
   = (0, 2 (a \mul\alpha a')).
\end{equation}

\noindent Note that if $M$ and $W$ are free $R$-modules, 
then $\sG_\alpha$ is an affine group scheme over $R$.

\begin{lemma}
  \label{lem:concnt_Galpha}
  Suppose that $\sG_\alpha(A)$ is finite.
  Then
  $\concnt(\sG_\alpha(A)) = {\card{W_A}} \dtimes \ask{2\alpha^A}$.
\end{lemma}
\begin{proof}
  $\card{\sG_\alpha(A)} \dtimes \concnt(\sG_\alpha(A))
     = \sum\limits_{(a, y) \in \sG_\alpha(A)} \card{\Cent_{\sG_\alpha(A)}(a, y)}
     = \card{W_A}^2 \sum\limits_{a\in M_A} \card{\Ker(a (2\alpha^A)}$.
\end{proof}

\begin{cor}
  \label{cor:Galpha_cc_zeta}
  Let $\fO$ be a compact \DVR{} with residue field of size $q$.
  Let $M \xto\alpha \Hom(M,W)$ be an alternating module representation over $\fO$.
  Suppose that $M$ and~$W$ are free of finite ranks $\ell$ and $e$, respectively.
  Then $\displaystyle \zeta_{\sG_\alpha}^{\cc}(s) = \zeta_{2\alpha}(s-e)$. \qed
\end{cor}

Hence, if $q$ is odd, then, up to integral shifts, 
every ask zeta function associated with an alternating module representation
(involving finitely generated free modules) over $\fO$ {is} a conjugacy class
zeta function;
note that Corollary~\ref{cor:cc_zeta_gpscheme} provides a partial converse.

\begin{ex}
  \label{ex:lins1}
  Let $\fO$ be a compact \DVR{} with residue field of odd size $q$.
  Write $V = \fO^d$.
  Let $V \xto\alpha \Hom(V,V\wedge V)$ be the alternating module
  representation over $\fO$ corresponding to the natural map
  $V\otimes V \to V \wedge V$ via the tensor-hom adjunction.
  Then $\sG_\alpha$ is a group scheme analogue of
  (the pro-$p$ completions of) the nilpotent groups of ``type $F$'' in the
  terminology of Stasinski and Voll~\cite[\S 1.3]{SV14}.
  Clearly, $\alpha^\MW$ is isotopic to $\So_d(\fO) \incl \Mat_d(\fO)$.
  By Corollary~\ref{cor:Galpha_cc_zeta},
  Corollary~\ref{cor:ask_zeta_duality}, and \cite[Prop.~5.11]{ask},
  $\Zeta_\alpha(T) = \Zeta_{\alpha^\MW}(T) =
  \Zeta_{\Mat_{d\times(d-1)}(\fO)}(T) = \frac{1-q^{1-d}T}{(1-T)(1-qT)}$
  whence
  \begin{equation}
    \label{eq:lins1}
  \zeta_{\sG_\alpha}^\cc(s) = \zeta_{\alpha}\left(s - \binom d 2\right)
  =
  \frac{1-q^{\binom {d-1} 2 - s}}{
    \Bigl(1-q^{\binom d 2 - s}\Bigr)
    \Bigl(1-q^{\binom d 2 +1- s}\Bigr)
  }.
  \end{equation}
  This was first proved by Lins~{\cite[Cor.\ 1.5, (1.4)]{Lin18b}}
  by computing a $\fP$-adic integral.
  In view of \cite[\S 5]{ask} and the preceding example,
  it seems remarkable how often natural instances of ask and conjugacy class
  zeta functions are of ``constant rank type'' in the sense that they arise by
  a specialisation $T \to q^n T$ of the formula in \eqref{eq:Kmin}
  (for $m=1$); see \S\ref{ss:cc_crk} below.
\end{ex}

\paragraph{Lie algebras.}
For a Lie-theoretic interpretation of $\sG_\alpha$,
let $\fg$ be a nilpotent Lie algebra of class $c$ over a ring $R$.
Suppose that multiplication by $c!$ is injective on $\fg$ and
that $[\fg,\fg]\subset c! \fg$.
By the Lazard correspondence, the Baker-Campbell-Hausdorff series
endows $\fg$ with the structure of a nilpotent group $\exp(\fg)$ of
class~$c$.

Let $M\xto\alpha\Hom(M,W)$ be an alternating module representation over $R$.
For $A \in \OBJ(\CALG(R))$,
we endow $\fg_\alpha(A) := M_A \oplus W_A$ with the structure of a nilpotent Lie $A$-algebra of
class at most $2$ by defining, for $a,a' \in M_A$ and $y,y'\in W_A$, 
$[ (a,y), (a', y')] = (0, a\mul\alpha{a'})$.

\begin{cor}
  \label{cor:altask_concnt}
  Let $M\xto\alpha \Hom(M,W)$ be an alternating module representation
  over~$R$.
  Suppose that $2 \in R^\times$ and
  let $A\in\OBJ(\CALG(R))$. Then:
  \begin{enumerate}
  \item
    \label{prop:altask_concnt1}
    $\sG_\alpha(A) \approx \exp(\fg_{2\alpha}(A))$.
  \item
    \label{prop:altask_concnt2}
    If $\sG_\alpha(A)$ is finite, then $\concnt(\sG_\alpha(A)) = \ask{\ad_{\bm\fg_{\alpha}(A)}}$.
  \end{enumerate}
\end{cor}
\begin{proof}
  Part (\ref{prop:altask_concnt1}) follows by inspection from the
  Hausdorff series
  $\log(\exp(X)\exp(Y)) = X + Y + \frac 1 2[X,Y] + \dotsb$
  and (\ref{prop:altask_concnt2}) follows from Proposition~\ref{prop:ask_concnt}.
\end{proof}

\begin{rem}
    Identical or similar constructions of group functors attached to
    alternating module representations have often appeared in the
    literature; see e.g.\ \cite[\S 1.3]{SV14}.
\end{rem}

\subsection{Alternating hulls of module representations}

Recall the notation from \S\ref{ss:addition}.
Give any module representation $M\xto\theta\Hom(V,W)$ over $R$,
its \emph{alternating hull} $\sk\theta$ is the composite
\[
V \oplus M \xto\tau M \oplus V \xto{\theta\oplus(-\theta^\MV)} \Hom(V\oplus M,
W\oplus W) \xto{\Hom(V\oplus M,\Sigma_W)} \Hom(V\oplus M, W),
\]
where $V \oplus M \xto\tau M \oplus V$ is the canonical isomorphism.
Hence, for $a,a'\in M$ and $x,x' \in V$,
\[
(x,a) \mul{\sk\theta}(x',a') = x\mul\theta a' - x'\mul\theta a.
\]
Note that $\sk{\dtimes\,}$ commutes with base change:
$\sk{\theta}^\phi =
\sk{\theta^\phi}$ for each ring map $R\xto\phi \tilde R$. 

\begin{rem}
  \label{rem:sk}
  \quad
  \begin{enumerate}
  \item
    \label{rem:sk1}
    $\sk\theta$ coincides with Wilson's bimap $\#$ derived from a bimap
    $*$ in \cite[\S 9.2]{Wil17}.
  \item
    \label{rem:sk2}
    Suppose that $M$, $V$, and $W$ are free $R$-modules of finite ranks
    and choose bases $\cA$, $\cX$, and $\cY$, respectively, of these modules.
    Recall the notation for matrices associated with module representations
    from \S\ref{ss:matrices}.
    After relabelling, we may assume that $M(\theta;\cA,\cX,\cY)$ and $M(\theta^\MV;\cX,\cA,\cY)$
    are matrices of linear forms in disjoint sets of variables.
    We may then identify
    \[
    M(\sk\theta;\cA\!\amalg\!\cX,\cA\!\amalg\!\cX,\cY) = 
    \begin{bmatrix}
      \phantom.M(\theta;\cA,\cX,\cY) \\
      -M(\theta^\MV;\cX,\cA,\cY)
    \end{bmatrix}.
    \]
  \end{enumerate}
\end{rem}

\subsection{Class numbers and general module representations}
\label{ss:general_gps}
\label{ss:gps_from_mreps}

For a module representation $\theta$ over $R$, we construct 
group functors whose associated class numbers are related to average sizes of kernels derived from $\theta$.

\paragraph{First construction.}
Let $M\xto\theta\Hom(V,W)$ be a module representation over $R$.
By Lemma~\ref{lem:concnt_Galpha},
\begin{equation}
  \label{eq:concnt_Gsk}
  \concnt(\sG_{\sk\theta}(A)) = \card{W_A} \dtimes \ask{2\sk{\theta^A}}
\end{equation}
for $A\in \CALG(R)$ whenever the group on the left-hand side is finite.

\paragraph{Second construction.}
The following construction appeared in work of Grunewald and
O'Halloran~\cite[Prop.\ 2.4]{GO'H85},
Boston and Isaacs~\cite[\S 2]{BI04}, and Wilson~\cite[\S 9]{Wil17}
(where it was referred to as ``folklore'').
Let $M\xto\theta \Hom(V,W)$ be a module representation over $R$.
The abelian group $M$ acts on $V\oplus W$ via
\begin{align*}
  (x,y). a & := (x,x (a\theta) + y) & (a\in M, x\in V,y\in W)
\end{align*}
and this action extends to define an action of $M_A$ on $(V\oplus
W)_A = V_A \oplus W_A$ for each $A\in \CALG(R)$.
Let $\sH_\theta(A) = M_A \ltimes (V\oplus W)_A$ be the
associated semidirect product and extend $\sH_\theta$ to a functor
$\CALG(R)\to\GRP$ in the evident way.

For the purpose of relating class numbers and average sizes of kernels, 
whenever $2$ is invertible in $R$ and in view of \eqref{eq:concnt_Gsk},
the group functors $\sG_{\sk\theta}$ and $\sH_\theta$ are interchangeable:

\begin{lemma}
  \label{lem:concnt_Htheta}
  Let $A \in \OBJ(\CALG(R))$ and suppose that
  $\sH_\theta(A)$ is finite.
  Then $$\concnt(\sH_\theta(A)) = \card{W_A} \dtimes \ask{\sk{\theta^A}}.$$
\end{lemma}
\begin{proof}
  Identify $\sH_\theta(A) = M_A\times V_A\times W_A$ as sets.
  By a simple calculation (see \cite[Prop.~2.4]{GO'H85}),
  for $a,a'\in M_A$, $x,x'\in V_A$, and $y,y'\in W_A$,
   \[
   [(a,x,y),(a',x',y')]_{\sH_\theta(A)}
   = (0,0,x (a'\theta) - x'(a\theta))
   = (0,0,(x,a) \mul{\sk\theta}(x',a')).
   \]
   The rest of the proof is then analogous to that of Lemma~\ref{lem:concnt_Galpha}.
\end{proof}

\begin{cor}
  \label{cor:ask_zeta_GH}
  Let $\fO$ be a compact \DVR{}.
  Let $M\xto\theta\Hom(V,W)$ be a module representation over $\fO$.
  Suppose that $M$, $V$, and $W$ are free of finite ranks $\ell$, $d$, and $e$,
  respectively.
  Then $\zeta_{\sG_{\sk\theta}}^\cc(s) = \zeta_{2\sk\theta}(s-e)$ and
  $\zeta_{\sH_\theta}^{\cc}(s) = \zeta_{\sk\theta}(s-e)$. \qed
\end{cor}

\subsection{Average sizes of kernels of alternating hulls}

In order to make use of \eqref{eq:concnt_Gsk},
Lemma~\ref{lem:concnt_Htheta}, and Corollary~\ref{cor:ask_zeta_GH},
we need to understand the effect of the operation $\sk{\dtimes\,}$
 on average sizes of kernels.
The crucial observation here is the following.

\begin{thm}
  \label{thm:ask_sk}
  Let $R$ be a finite truncated \DVR{}.
  Let $M \xto\theta\Hom(V,W)$ be a module representation over $R$.
  Suppose that $M$, $V$, and $W$ are finite.
  Then
  \[
  \ask{\sk\theta} = \frac{\card M}{\card V} \dtimes \askm {\theta^\MW} 2.
  \]
\end{thm}
\begin{proof}
  By Theorem~\ref{thm:ask_duality}, $\ask{\sk\theta}$ coincides with
  $\ask{{\sk\theta}^\MW}$ so it suffices to compute the latter.
  By Corollary~\ref{cor:bullet_func}, ${\sk\theta}^\MW$ is the composite
  \[
  W^*
  \!\xto{\Delta_{W^*}}\!
  W^* \oplus W^*
  \!\xto{\theta^\MW \oplus (-\theta^{\MV\MW})}\!
  \Hom(V\oplus M, M^*\oplus V^*)
  \!\xto[\approx]{\Hom(V\oplus M,\tau^*)}\!
  \Hom(V\oplus M, V^*\oplus M^*).
  \]

  Hence, by Lemma~\ref{lem:askm_coll}
  and using $\theta^{\MV\MW} \approx (\theta^{\MW})^{\MW\MV\MW}
  \approx \theta^{\MW\VW}$  (see \S\ref{ss:S3})
  \[
  \ask{{\sk\theta}^\MW} =
  \frac 1{\card{W^*}}
  \sum_{\psi\in W^*}
  \card{\Ker(\psi\theta^\MW)} \dtimes \card{\Ker(\psi \theta^{\MV\MW})}
  =
  \frac 1{\card{W^*}}
  \sum_{\psi\in W^*}
  \card{\Ker(\psi\theta^\MW)} \dtimes \card{\Ker((\psi \theta^{\MW})^*)}.
  \]

  If $U_1, U_2$ are finitely generated $R$-modules 
  and $U_1 \xto\lambda U_2$ is a homomorphism, then
  $\Ker(\lambda^*) \approx \Coker(\lambda)^*\approx \Coker(\lambda)$ 
  (by Lemma~\ref{lem:tdvr}(\ref{lem:tdvr2})) and 
  $\card{\Ker(\lambda^*)} = \card{U_2/\Img(\lambda)} =
  \frac{\card{ U_2}} {\card{U_1}} \dtimes \card{\Ker(\lambda)}$.
  Therefore,
  \[
  \ask{{\sk\theta}^\MW} =
  \frac {\card {M^*}} {\card{V} \dtimes \card{W^*}}
  \sum_{\psi\in W^*}
  \card{\Ker(\psi\theta^\MW)}^2 = 
  \frac{\card {M}}{\card{V}} \dtimes \askm {\theta^\MW}2.\qedhere
  \]
\end{proof}

\begin{rem}
  In the setting of Remark~\ref{rem:sk}(\ref{rem:sk2}),
  \[
  M(\sk{\theta}^\MW;\cY^*,\cA \!\amalg\! \cX,\cA^*\!\amalg\! \cX^*) =
  \begin{bmatrix}
    0 & M(\theta^\MW;\cY^*,\cA,\cX^*) \\
    -M(\theta^\MW;\cY^*,\cA,\cX^*)^\top & 0
  \end{bmatrix};
  \]
  cf.\ the ``doubling'' operation of Ilic and Landsberg~\cite[\S 2.7]{IL99}.
\end{rem}

In particular, Theorem~\ref{thm:ask_sk} provides a new method
for computing \Askm 2 zeta functions:

\begin{cor}
  \label{cor:ask2_cc}
  Let $\fO$ be a compact \DVR{} with residue field of size $q$.
  Let $M\xto\theta\Hom(V,W)$ be a module representation over $\fO$.
  Suppose that $M$, $V$, and $W$ are free of finite ranks, $\ell$, $d$, and
  $e$, respectively. Then:
  \begin{enumerate}
  \item 
    \label{cor:ask2_cc1}
    $\Zeta^2_\theta(T) = \Zeta_{\sk{\theta^\MW}}(q^{d-\ell}T)$.
  \item
    \label{cor:ask2_cc2}
    $\displaystyle \zeta_{\sH_\theta}^\cc(s) = \zeta^2_{\theta^\MW}(s+d-e-\ell)$.
  \end{enumerate}
\end{cor}
\begin{proof}
  Combine Theorem~\ref{thm:ask_sk} and Corollary~\ref{cor:ask_zeta_GH}.
\end{proof}

\begin{ex}[$\Zeta^2_{\Mat_d(\fO)}(T)$]
  \label{ex:lins2}
  Let $\fO$ be a compact \DVR{} with residue field of odd size~$q$.
  Let $V = \fO^d$ and let $V^*\xto\theta\Hom(V,\End(V))$ be the
  module representation which corresponds to the natural isomorphism
  $V^*\otimes V \to \End(V)$ via the tensor-hom adjunction.
  Then $\sH_\theta$ is a group scheme analogue of (the pro-$p$ completions of)
  the nilpotent groups of ``type $G$'' in the sense of Stasinski and
  Voll~\cite[\S 1.3]{SV14}.
  Lins~\cite[Cor.~1.5]{Lin18b} showed that
  \begin{equation}
    \label{eq:lins_Mdxe}
    \zeta_{\sH_\theta}^\cc(s) = \frac
    {(1-q^{2\binom d 2 }t)(1-q^{2\binom d 2 + 1}t) + q^{d^2}t(1-q^{-d})(1-q^{1-d})}
    {(1-q^{d^2}t)^2(1-q^{d^2+1}t)},
  \end{equation}
  where $t = q^{-s}$.
  As we will now explain, we can read off the \Askm 2\! zeta function
  of (the identity map on) $\End(V)$ from Lins's formula~\eqref{eq:lins_Mdxe}.
  First, by Corollary~\ref{cor:ask2_cc}{(\ref{cor:ask2_cc2})},
  $\zeta_{\sH_\theta}^\cc(s) = \zeta_{\theta^\MW}^2(s-d^2)$.
  Next, $\theta^\MW$ is isotopic to 
  the identity $\End(V^*) \xlongequal{\phantom{??}} \End(V^*)$.
  Indeed, in the following commutative diagram, unlabelled maps are the evident
  isomorphisms:
  \[
  \begin{CD}
    \End(V^*) @= \Hom(V^*,V^*) \\
    @AAA @AAA\\
    (V^* \otimes V)^* @>>> V^* \otimes V^{**}\\
    @VVV @VVV\\
    \End(V)^* @>{\theta^{\MW}}>> \Hom(V,V^{**}).
  \end{CD}
  \]
  We thus conclude from \eqref{eq:lins_Mdxe} that
  \begin{align*}
    \Zeta^2_{\Mat_d(\fO)}
    & = \Zeta^2_{\End(V^*)}(T) = \Zeta^2_\theta(T) 
      = \Zeta^2_{\theta^\MW}(T)
    \\
    & = \frac
    {(1-q^{-d}T)(1-q^{1-d}T) + T(1-q^{-d})(1-q^{1-d})}
    {(1-T)^2(1-qT)}.
  \end{align*}
\end{ex}

Curiously, the easiest way of computing \Askm 2{}\! zeta functions
often seems to be to pass to the alternating hull as in
Example~\ref{ex:lins2}.

\begin{qu}
  \label{qu:askm_Mdxe}
  What is the \Askm m\! zeta function of $\Mat_{d\times e}(\fO)$ for
  general $m$, $d$, and $e$?
\end{qu}

\begin{ex}[Smooth determinantal hypersurfaces---revisited]
  Let $\fO$ be a compact \DVR{} with residue field $\RF$ of size $q$.
  Let $a_1,\dotsc,a_\ell \in \Mat_d(\fO)$.
  Write $\XX = (X_1,\dotsc,X_\ell)$, $a(\XX) = X_1 a_1 + \dotsb + X_\ell a_\ell \in
  \Mat_d(\fO[\XX])$, and $F = \det(a(\XX))$.
  Let $\fO^\ell \xto{\theta} \Mat_d(\fO)$ be the module representation
  $x\theta = a(x)$ ($x \in \fO^\ell$).
  Then, by Corollary~\ref{cor:ask2_cc}(\ref{cor:ask2_cc2}),
  \[
  \zeta_{\sH_{\theta^\MW}}^{\cc}(s) = \zeta_\theta^2(s-\ell).
  \]

  Suppose that for all $\bar \xx\in \RF^\ell$, if
  $F(\bar\xx) = \frac{\partial F (\bar\xx)}{\partial X_1} = \dotsb =
    \frac{\partial F (\bar\xx)}{\partial X_\ell} = 0$, then $\bar\xx = 0$.
  Let $H := \mathrm{Proj}\!\left(\fO[\XX]/F\right)$.
  Then a straightforward variation of the case $m = 1$ in
  \cite[Thm~7.1]{ask} shows that

  \vspace*{-.4em}
  {\small\[
    \Zeta^m_\theta(T) = \frac{1-q^{-\ell}T}{(1-T)(1-q^{dm-\ell}T)}  + \#H(\RF) (q-1)
    q^{m-\ell}T \dtimes \frac{1 - q^{-m}}{(1-T)(1-q^{m-1}T)(1-q^{dm-\ell}T)}.
    \]}

  The coefficient of $q^{-s}$ in the Dirichlet series
  $\zeta_{\sH_{\theta^\MW}}^{\cc}(s)$ can be read off from~\cite[Thm~4.1]{O'BV15}.
\end{ex}

\subsection{Further conjugacy class zeta functions of ``constant rank type''}
\label{ss:cc_crk}

Define $c_d(X_1,\dotsc,X_d) \in \Mat_{\binom{d+1} 2 \times
  d}(\ZZ[X_1,\dotsc,X_d])$ via
\[
c_d(X_1,\dotsc,X_d) = \begin{bmatrix}
  X_1 1_d \\
  \hline
  \begin{array}{c|c}
    0_{\binom d 2 \times 1} & c_{d-1}(X_2,\dotsc,X_d)
    \end{array}
\end{bmatrix}.
\]
Let $\ZZ^d \xto{\gamma_d} \Mat_{\binom{d+1} 2\times d}(\ZZ)$ be the
module representation $x\gamma_d = c_d(x)$.

\begin{prop}
  \label{prop:gamma_cc}
  Let $\fo$ be the ring of integers of a number field $k$.
  Let $\zeta_k(s)$ denote the Dedekind zeta function of $k$.
  Then
  \[
  \zeta_{\sH_{\gamma_d^\fo}}^{\cc}(s)
  =
  \frac
  {\zeta_k\left(s-\binom{d+1}2-d\right)\zeta_k\left(s-\binom{d+1}2-d+1\right)}
  {\zeta_k\left(s-\binom{d+1}2 + 1\right)}.
  \]
\end{prop}

Let $\fO$ be a compact \DVR{}.
Prior to proving Proposition~\ref{prop:gamma_cc}, 
we first compute the $\Askm m$\! zeta function of $\theta_d := \gamma_d^{\fO}$.
\begin{lemma}
  \label{lem:Zm_gamma}
  $\displaystyle
  \Zeta_{\theta_d}^m(T) = \frac 1 {1 - q^{m\binom {d+1} 2 -d}T}
  \dtimes
  \prod_{j=0}^{d-1}
  \frac {1 - q^{m\binom d 2 + (m-1)j - 1}T}
  {1 - q^{m\binom d 2 + (m-1)j}T}.
  $
\end{lemma}
\begin{proof}
  Recall the definition of $\imgsize_\theta$ from \S\ref{ss:ask_zeta_duality}.
  For $x\in \fO^d$ and $y\in\fO\setminus\{0\}$, we have
  $\imgsize_{\theta_d}(x,y) = \abs y^{-d} \prod_{i=1}^d
  {\norm{x_1,\dotsc,x_i,y}}$ (cf.\ \cite[Lem.\ 4.6]{ask})
  whence
  $$\kersize_{\gamma_d^\fO}(x,y) = \frac 1{\abs{y}^{\binom d 2} \prod\limits_{i=1}^d \norm{x_1,\dotsc,x_i,y}}.$$
  The claim now follows from Theorem~\ref{thm:int} and \cite[Lem.\ 5.8]{ask}.
\end{proof}

\begin{rem}
  $\gamma_d^\MV$ is isotopic to the inclusion
  of upper triangular matrices into $\Mat_d(\ZZ)$.
  The case $m = 1$ of Lemma~\ref{lem:Zm_gamma} thus follows
  from \cite[Prop.~5.15(ii)]{ask} and Corollary~\ref{cor:ask_zeta_duality}.
\end{rem}

\begin{proof}[Proof of Proposition~\ref{prop:gamma_cc}]
  For $m = 2$, Lemma~\ref{lem:Zm_gamma} simplifies to
  \[
  \Zeta^2_{\theta_d}(T) = \frac{1 - q^{2\binom d 2 - 1}T}
  {\left(1 - q^{\binom{d+1} 2 + \binom d 2}T\right)\left(1 - q^{\binom{d+1} 2 + \binom d 2-1}T\right)}.
  \]
  We leave it to the reader to verify that $\gamma_d^\MW \approx \gamma_d$.
  Thus, by Corollary~\ref{cor:ask2_cc}, writing $t = q^{-s}$,
  \begin{equation}
    \label{eq:zeta_H_gamma}
    \zeta_{\sH_{\theta_d}}^\cc(s)
    =
    \zeta^2_{\theta_d}\left(s + \binom d 2 - d\right)
    =
    \frac{1- q^{\binom{d+1} 2 -1}t}
    {(1 - q^{\binom{d+1}2 +d}t)(1-q^{\binom{d+1}2 + d-1}t)}.
  \end{equation}
  The claim follows by taking the Euler product over all $\fO = \fo_v$
  for all places $v\in\Places_k$.
\end{proof}

Curiously, \eqref{eq:zeta_H_gamma} and \cite[Prop.\ 1.5]{ask} (see
Example~\ref{ex:another_Mdxe}) show that 
\begin{equation}
  \label{eq:zeta_H_gamma_vs_Mdxe}
\zeta_{\sH_{\theta_d}}^\cc\left(s + \binom{d+1}2 + d\right)
=
\frac{1 - q^{-(d+1)}t}
{(1-t)(1-q^{-1}t)}
= \zeta_{\Mat_{d\times(d+1)}(\fO)}(s),
\end{equation}
where $t = q^{-s}$.
While \eqref{eq:zeta_H_gamma_vs_Mdxe} is reminiscent of \eqref{eq:lins1}, the author cannot at present
explain the former numerical coincidence conceptually.

{
  \bibliographystyle{abbrv}
  \bibliography{ask2}
}

\end{document}